\documentclass[a4paper,12pt]{article}

\usepackage[utf8]{inputenc}
\usepackage[T1]{fontenc}
\usepackage{mathtools,url}
\usepackage[left=2cm,top=2.5cm,right=2cm,bottom=2.5cm]{geometry}

\usepackage{pgf,tikz,pgfplots}
\usepackage{mathrsfs}
\usetikzlibrary{arrows}

\usepackage{setspace}
\usepackage{multicol}
\usepackage{cite}
\usepackage{amsthm}
\usepackage{amsmath,amssymb}
\usepackage{enumerate}
\usepackage{comment}

\usepackage{esint}

\usepackage{appendix}

\usetikzlibrary{hobby}
\usepgfplotslibrary{fillbetween}

\newtheorem{theo}{Theorem}[section] 
\newtheorem*{theo*}{Theorem}
\newtheorem{lem}[theo]{Lemma} 
\newtheorem{prop}[theo]{Proposition}
\newtheorem{cor}[theo]{Corollary}
\newtheorem{rem}{Remark}[section]
\newtheorem{definition}{Definition}[section]

\newcommand{\R}{\mathbb{R}}\newcommand{\Rn}{\R^n}\newcommand{\Rnn}{\R^{n\times n}}
\newcommand{\N}{\mathbb{N}}
\DeclareMathOperator{\supp}{supp}
\DeclareMathOperator{\diver}{div}
\DeclareMathOperator{\curl}{curl}

\DeclareMathOperator{\tr}{tr} 
\DeclareMathOperator{\pv}{pv}
\DeclareMathOperator{\dist}{dist}
\DeclareMathOperator{\loc}{loc}

\renewcommand{\O}{\Omega}
\renewcommand{\a}{\alpha}
\renewcommand{\b}{\beta}
\renewcommand{\d}{\delta}
\renewcommand{\l}{\lambda}

\newcommand{\p}{\partial}

\newcommand{\e}{\varepsilon}

\newcommand{\f}{\varphi}

\newcommand{\qqquad}{\qquad\qquad}

\usepackage{amsmath}
\usepackage{latexsym}
\usepackage{amssymb}

\def\XXint#1#2#3{{\setbox0=\hbox{$#1{#2#3}{\int}$}
		\vcenter{\hbox{$#2#3$}}\kern-.5\wd0}}

\usepackage{graphicx}
\usepackage{wrapfig}
\begin{document}
	\title {Nonlocal Green theorems and Helmholtz decompositions for truncated 
	fractional gradients}

	\author{Jos\'e Carlos Bellido \and Javier Cueto \and Mikil D. Foss \and Petronela Radu}
%

\date{}


	
	

\maketitle	

	\begin{abstract}
 In this work we further develop a nonlocal calculus theory (initially introduced in \cite{BeCuMC22}) associated with singular fractional-type operators which exhibit kernels with finite support of interactions. The applicability of the framework to nonlocal elasticity and the theory of peridynamics has attracted increased interest and motivation to study it and find connections with its classical counterpart. In particular, a critical contribution of this paper is producing vector identities, integration by part type theorems (such as the Divergence Theorem, Green identities), as well as a Helmholtz-Hodge decomposition. The estimates, together with the analysis performed along the way provide stepping stones for proving additional results in the framework, as well as pathways for numerical implementations. 

	\end{abstract}



	\section[Introduction]{Introduction}

	In the last couple of decades there has been a significant increase in employing nonlocal frameworks for modeling real phenomena. One such example is the fractional Laplacian, possibly the most famous nonlocal operator in the literature, obtained as a generalization of the classical Laplacian to non-integer differentiability indices. The growing prominence of nonlocal models is due to their ability to handle less regular functions and accommodate long range interaction forces. This is of particular interest for systems that model phenomena with singularities. A prototypical example in this direction is the theory of peridynamics introduced in \cite{Silling2000} to model dynamic fracture. 
	
 Different modeling purposes dictate that physical quantities must capture information around a point, so to this end we consider one-point nonlocal gradients, defined through an integral. A specific structure for one-point nonlocal gradients, inspired by the theory of state-based peridynamics \cite{silling2010linearized} is considered in \cite{MeS,Med,MeD16,DeGuOlKa21} and is given by:
	\begin{equation*}
		\mathcal{G}_\rho u(x)=\int \frac{u(x)-u(y)}{|x-y|}\frac{x-y}{|x-y|} \rho(x-y)\, dy
	\end{equation*}
	with $\rho \in L^1(\Rn)$. 
	
	The choice of the interaction kernel will determine the analytical properties of the framework. In \cite{DuTi18},  $\tilde{\rho}(x)=\frac{\rho(x)}{|x|}$ was chosen symmetric and non-integrable in order to enjoy a Poincaré inequality (thus avoiding instabilities in the Dirichlet energy). This feature is still shown to continue to hold if symmetry is abandoned for integrable kernels $\widetilde{\rho}$ (see \cite{FossPoincare}), and in particular for kernels supported over half-balls as in \cite{HanTian}.
	
	In this paper we focus on symmetric kernels as they lead to symmetry of second derivatives and vector identities similar to classical calculus (e.g. the curl of the gradient vanishes, etc.). Moreover, symmetric kernels are also predominantly employed in physical models. One particular operator with a symmetric kernel is the Riesz fractional gradient, where in the formulation above we choose $\rho(x)=c_{n,s}\frac{1}{|x|^{n-1+s}}$ and the integration is performed over $\Rn$ \cite{ShS2015,ShS2018}. 
	In order to consider systems on bounded domains, as they predominantly appear in applications, here we  consider a smoothly truncated version of this operator, whose functional space domain and minimization results were studied in \cite{BeCuMC22b,CuKrSc23,BeCuMC22}. It has been shown in these works that the resulting framework is suitable for nonlocal hyperelasticity. 
	
	To be more precise, the functional spaces with which we will work here are defined as the closure of smooth functions under the norm given by
	\begin{equation*}
		\|u\|_{H^{s,p,\d}(\O)}=\|u\|_{L^p(\O_\d)} + \|D^s_\d u\|_{L^p(\O)},
	\end{equation*}
 where the spaces and domains above will be introduced in the sequel. While the functional domains will require slightly more than $L^p$ integrability, this may be sufficient for modelling purposes aiming to include singularity phenomena. Indeed, jump functions representing discontinuous phenomena along a hypersurface belong to these spaces as long as $sp<1$, whereas discontinuous functions at a point, such as $\frac{x}{|x|^n}$, are found in these spaces for $sp<n$ (see the discussion in \cite[Section 6]{BeCuMC22b}). As these spaces are initially defined as the closure of smooth functions under a natural norm, nonlocal gradients of non-smooth functions are defined as limits of gradients of smooth ones. However, thanks to \cite[Lemma 4.2]{BeCuMC22b} or \cite[Theorem 2.13]{CuKrSc23}, we have a more convenient formulation given by the following formula
	\begin{equation*}
		D^s_\d u(x)= \nabla \int_{B(x,\d)} u(y) Q^s_\d (x-y) \, dy
	\end{equation*} 
	where $Q^s_\d(x)=\displaystyle\frac{q_\d(x)}{|x|^{n-1+s}}$, a smoothly truncated Riesz potential, is an integrable function that plays the role of the latter in this framework over bounded domains. We use $\nabla$ to denote the standard gradient operator, and we will refer to $\delta$ as the horizon of interaction.

\subsection{Motivation for the framework. Literature review} 

Nonlocal frameworks are able to capture complex, long-range interactions that traditional local models cannot adequately represent. Among their defining characteristics, nonlocal models have the advantage of collecting information in a neighborhood, a physical trait that is essential in some applications, such as dynamic fracture or collective behavior. Finally, the solution spaces for nonlocal systems are considerably larger, allowing for more physically salient inputs. 

As mentioned above, the protagonist operators in this paper are one-point nonlocal gradients, defined through an integral, motivated by the theory of state-based peridynamics. In order to endow the framework with tools such as a nonlocal fundamental theorem of calculus, compact embedding theorems, or a Poincaré-Sobolev inequality, the kernel must exhibit a strong singularity, so the operators behave like fractional derivatives. As the infinite support of the kernel is unphysical and creates difficulties in analysis and numerical simulations, we will set the analysis in the framework established in \cite{BeCuMC22} that exhibits smoothly truncated fractional gradients. A particular advantage offered by compactly supported kernels is seen in tackling boundary value problems on bounded domains, in particular, imposing boundary conditions on sets of finite measure. The framework was shown to fit well with nonlocal hyperelasticity \cite{BeCuMC22b}, moreover, the set of tools and properties that it enjoys renders it as a versatile setting for a wide variety of models.

Among the results obtained in this paper, we prove a nonlocal version of the Helmholtz-Hodge decomposition, an important theoretical result that in the classical framework has vast applicability. Obtaining this type of decomposition has the potential to advance the understanding and applicability of nonlocal models, so there has already been some progress in this direction, with more advances forthcoming. Previously, nonlocal Helmholtz decompositions were obtained on bounded domains for frameworks that involve: two-point gradients \cite{delia2020helmholtz}; one-point convolution gradients $u*\widetilde{\rho}$ with $\widetilde{\rho}\in L^1(\Rn)$ \cite{haarradu}; half-ball supported nonlocal gradients \cite{HanTian}; and Riesz fractional gradients \cite{delia2021connections}. As it was mentioned above, here we consider a nonlocal gradient with a fractional singularity but defined over bounded domains. This will give rise to a decomposition with boundary terms over a collar of width $2\d$, as the Euler-Lagrange equations obtained from the variational problems \cite{BeCuMC22,CuKrSc23,BeCuMC22b} involve terms where two nonlocal derivatives are taken.

As it has previously been mentioned, the framework based on the gradient $D^s_\d$ was first introduced in \cite{BeCuMC22,Cue21}, developing, in particular, the necessary functional analysis results necessary to prove the existence of minimizers of convex energy functionals such as continuous and compact embeddings and, in particular, nonlocal Poincar\'e inequality. All of these were obtained thanks to a nonlocal version of the fundamental theorem of calculus, a representation result showing that smooth functions can be recovered through convolution from their nonlocal gradients. The development of this theory has continued by proving existence of minimizers of polyconvex energy functionals (through a nonlocal Piola Identity and weak continuity of $\det D^s_\d u$) in \cite{BeCuMC22b} as well as quaxiconvex energy functionals in\cite{CuKrSc23}. The latter was obtained through some formulas relating classical and nonlocal gradients, and complemented with homogenization and relaxation results, as well as the continuity on $s$ of the functionals, including as a particular case the localization when $s$ goes to $1$. The Euler-Lagrange equations obtained in \cite{BeCuMC22b,BeCuMC22} were linearised in \cite{BeCuMC23}, where the resulting model is well posed thanks to a nonlocal Korn inequality and was shown to coincide with the nonlocal Eringen model under the proper choice of the kernel. 

Another point of interest in our investigation is the nonlocal Laplacian associated to our operator and its fundamental solution. Fundamental solutions for nonlocal Laplacians have been obtained for formulations with integrable kernels \cite{radu2019doubly}, \cite{haarradu}, as well as for Riesz-fractional ones \cite[Th. 2.8]{Bucur2016}, \cite[Th. 6.1]{delia2021connections}. 

	\subsection{Main contributions and organization of the paper}

 The nonlocal framework in which the results of this paper are set has been first introduced in \cite{Cue21}. As it was shown, these nonlocal operators can be valuable tools in elasticity, so we continue the analytical study of this framework here. The list below offers a summary of this paper's contributions.
 \begin{enumerate}[(i)]
 \item In \cite{Cue21, BeCuMC22, BeCuMC22b} the nonlocal divergence, gradient, and Laplacian were introduced. Here we introduce the nonlocal curl which, together with the other operators, will form the basis of the nonlocal vector calculus theory that we will study.   
 \item Following the roadmap of classical calculus, we prove here nonlocal versions of vector calculus identities, the Divergence Theorem, integration by parts (introduced here with boundary terms and for more general functions), as well as nonlocal counterparts to the three Green identities.
 \item We derive the fundamental solution to the nonlocal Laplacian $\Delta^s_\delta$ and study its behavior. In particular, we show that it behaves like the fundamental solution of the classical Laplacian at infinity and like the fundamental solution of the fractional Laplacian around the origin, thus showcasing the versatility of this nonlocal operator in capturing different behaviors. 
 \item We prove a Helmholtz-Hodge decomposition which shows that any vector field can be decomposed in nonlocal curl and divergence free components.  
 \item Boundary terms are identified - see Theorem \ref{th: NL Integration by parts}, Lemma \ref{cor: divergence for laplacian}, a first step towards understanding the choice of boundary conditions towards well-posedness, especially for Neumann-type problems. Formulating properly defined boundary conditions for nonlocal models is of critical importance in applications, as they ensure that the nonlocal model accurately represents the physical system under study, preventing unphysical behavior or numerical instability at the domain's edges. 
 \end{enumerate}


These results provide the foundations for studying nonlinear problems of physical applications, in continuum mechanics and other fields. Linear identities for the nonlocal operators (Prop 3.2) and the estimates obtained herein will serve as helpful tools in future investigations. 

\medskip

{\bf Organization of the paper.} The next section of the paper is dedicated to covering preliminaries for the results, setting up notation, definitions, and recalling some previously derived tools for this nonlocal calculus framework. Properties and identities for the theory are identified in Section 3, followed by The Divergence Theorem, Integration by Parts. In Section 4 the fundamental solution of the nonlocal Laplacian is derived, together with some properties on its behavior. The nonlocal Helmholtz decomposition is presented in Section 6, while the three nonlocal Green identities are stated and proved in Section 5. The final section of Conclusions precedes an Appendix containing some auxiliary results.

	\section[Preliminaries]{Preliminaries}\label{Prelim}
\subsection{Notation} \label{subsec: notation}
	We clarify here the convention we use for the Fourier transform. For $f\in L^1(\Rn)$ its Fourier transform is given by
 \begin{equation*}
    \widehat{f}(\xi)=\int_{\Rn}f(x)e^{-2\pi i \xi \cdot x}\, dx.
 \end{equation*}
 At times we may also use the notation $\mathcal{F}(f)=\widehat{f}$. When $f$ is in the Schwartz space $\mathcal{S}$, the Fourier transform defines an isomorphism which can be extended by continuity and duality to $L^2(\Rn)$ and further to the space of tempered distributions $\mathcal{S}'$. Notable references for Fourier analysis are \cite{duon2000,Grafakos08a}.

 Given the role that radial functions play in Section \ref{se: NL fundamental solution}, and since the Fourier transform of a radial (respectively, vector radial) function is also radial (respectively, vector radial), see, e.g., \cite[App.\ B.5]{Grafakos08a}, we recall these definitions gathered in \cite[Def. 2.1]{BeCuMC22}.
 \begin{definition}\label{de:radial}
	We will say that
	\begin{enumerate}[a)]
		\item a function $f:\Rn \rightarrow \R$ is \emph{radial} if there exists $\bar{f}:[0,\infty) \to  \R$ such that $f(x)=\bar{f}(|x|)$ for every $x \in \Rn$. In such a case, $\bar{f}$ is the radial representation of $f$. 
		\item a radial function  $f:\Rn \rightarrow \R$ is \emph{radially decreasing} if its radial representation $\bar{f}:[0,\infty) \to  \R$ is a decreasing function.
		\item a function $\f:\Rn \rightarrow \Rn$ is \emph{vector radial} if there exists a radial function $\bar{\f}:[0,\infty)\to  \R$ such that $\f(x)=\bar{\f}(|x|)x$ for every $x \in \Rn$.
	\end{enumerate}
\end{definition}
Regarding the reflection of a function we will use the notation $\widetilde{f}(x):= f(-x)$. As for the translation by a fixed $y$, we will write  $f_y(x)=f(x-y)$, so that for a function whose variable is $x$ the nonlocal derivatives will be taken with respect to this variable.
 
 \subsection[Nonlocal operators and functional spaces]{Nonlocal operators and functional spaces}
	This section is dedicated to the introduction of the nonlocal operators which we will employ, as well as some background results that were obtained previosuly.
	
Nonlocal gradients are generally based on a nonlocal kernel. For this analysis we will use the one from \cite{BeCuMC22}. To wit, 
	\begin{equation}\label{kernel}
		\rho_\delta(x)=\frac{1}{\gamma (1-s)|x|^{n-1+s}}w_{\d}(x),
        \quad\text{ with }\quad
        \gamma(s)=\frac{\pi^{\frac{n}{2}}2^s\Gamma\left(\frac{s}{2}\right)}{\Gamma\left(\frac{n-s}{2}\right)}.
	\end{equation}
	Here $\Gamma$ is the Euler's gamma function, $\delta>0$ is the horizon of interaction (as introduced in the theory of peridynamics \cite{Silling2000}), and $w_\d$ is a nonnegative cut-off function.  More precisely:
 
 
	\begin{enumerate}[a)]
		\item $w_\delta:\Rn\to[0,\infty)$ is radial; i.e. $w_\d(x)=\bar{w}_\d(|x|)$ for some nonnegative $\bar{w}_\d\in C_c^\infty([0,\infty))$, with $\supp(\bar{w}_\d)\subset [0,\delta)$.
		\item There is a constant $0<b_0<1$ such that $\bar{w}|_{[0,b_0\d]}=a_0$, where $a_0=\max_{r\ge 0}\bar{w}_\d(r)$.
		\item
		$\bar{w}_\d (r_1) \geq \bar{w}_\d (r_2)$ whenever $r_1 \leq r_2$.
	\end{enumerate}
Note that under these assumptions the kernel $\rho_\d$ defined in \eqref{kernel} is integrable, i.e. $\rho_\d \in L^1 (\Rn)$.

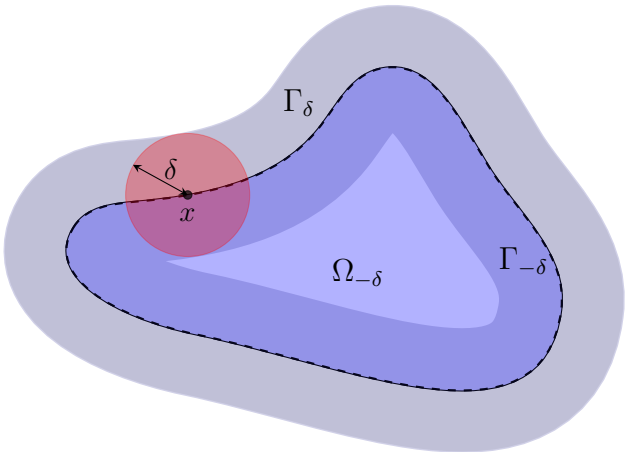
\begin{wrapfigure}{l}{.5\textwidth}
\begin{tikzpicture}[scale=1.6,use Hobby shortcut]
\clip (-.6,-.5)rectangle(4.7,3.25);
\draw[name path=dom,white,line width=45.5pt,closed,xshift=2pt,yshift=-3pt]
  (0,1.5)..(1,.6)..(4,.8)..(3.4,2)..(2.5,2.8)..(2,2.3)..(.5,1.7);
\draw[name path=dom,black!40,opacity=.2,line width=46pt,closed,xshift=2pt,yshift=-3pt]
  (0,1.5)..(1,.6)..(4,.8)..(3.4,2)..(2.5,2.8)..(2,2.3)..(.5,1.7);
\draw[name path=dom,line width=.8pt,closed,xshift=2pt,yshift=-3pt]
  (0,1.5)..(1,.6)..(4,.8)..(3.4,2)..(2.5,2.8)..(2,2.3)..(.5,1.7);
\draw[name path=split line, line width=0pt, opacity=0,yshift=-3pt](0,1.5)..(5,1.5);
\tikzfillbetween[of=dom and split line]{blue!30};
\draw[name path=dom,line width=.8pt,dashed,closed,xshift=2pt,yshift=-3pt]
  (0,1.5)..(1,.6)..(4,.8)..(3.4,2)..(2.5,2.8)..(2,2.3)..(.5,1.7);

\draw[name path=dom,black!50!blue!90,opacity=.8,line width=47pt,closed,xshift=2pt,opacity=.2,yshift=-3pt]
  (0,1.5)..(1,.6)..(4,.8)..(3.4,2)..(2.5,2.8)..(2,2.3)..(.5,1.7);
\node[above left,inner sep=0pt] (gamma) at (2.05,2.3){$\Gamma_\delta$};
\node at (2.4,1){$\O_{-\d}$};
\node at (3.76,1.1){$\Gamma_{-\d}$};

\coordinate (x) at (1,1.65);
\coordinate (y) at (.55,1.9);

\draw[red,fill,opacity=.3] (x) circle(.51);
\draw[black,fill,opacity=.7] (x) circle(.035);
\node[below,yshift=-1.5pt,scale=.9] at (x){$x$};
\draw[stealth-stealth,thin] (x) to node[xshift=4,yshift=4,scale=.9]{$\delta$} (y);

\end{tikzpicture}

\caption{$\Gamma_\delta$ (gray); $\Gamma_{-\d}$ (dark blue); and $\O_{-\d}$ (light blue)}
  \end{wrapfigure}
  
The focus of this paper is to continue developing a framework of nonlocal operators compatible with bounded domains, which was initially introduced in \cite{BeCuMC22}. Throughout the article, $\O\subset \Rn$ is a bounded open set. To ensure its nonlocal gradient is well-defined on $\O$, a function needs to be defined on the enlarged domain, which we decompose using an exterior collar:
 \[
    \O_\d:=\O + B(0,\d)\quad\text{ and }\quad\Gamma_\delta=\O_\delta\setminus\O.
\]
We use $B(x,\delta)\subseteq\Rn$ for the open ball with radius $\delta>0$ and center $x\in\Rn$. At times, we will also need to decompose $\O$ using an inner collar: $\O_{-\d}:=\{ x\in \O : \dist(x,\partial \O)> \d \}$ and $\Gamma_{-\delta}=\O\setminus\O_{-\delta}$. Given $0<\e<\d$, we also define $\Gamma_{\pm\e}:=\{x \in \O_\d : \dist(x,\partial \O) < \e \}=\O_\e\backslash\O_{-\e}=\Gamma_\e \cup \Gamma_{-\e}$. Without explicitly stating it, we may extend the domain of a compactly supported function by zero to $\Rn$.
 


Next, we introduce the definitions for the nonlocal versions of the gradient, divergence, curl and Laplacian for our framework.	
	\begin{definition} \label{def: nonlocal gradient}
		Let $0<s<1$, $0<\d$ and set
		\begin{equation*}
			c_{n,s}:= \frac{n+s-1}{\gamma(1-s)}.
		\end{equation*}
		\begin{enumerate}[a)]
			\item \label{item:Dsdu} Let $u\in C_c^{\infty} (\Rn)$. Then, for every $x \in \O$ the nonlocal gradient $D_\delta^s u$ is defined as
			\begin{equation} \label{eq: definition of nonlocal gradient}
				D_\delta^s u(x):= c_{n,s} \int_{B(x,\delta)} \frac{u(x)-u(y)}{|x-y|}\frac{x-y}{|x-y|}\frac{w_\d(x-y)}{|x-y|^{n+s-1}} \, dy.
			\end{equation}
			\item \label{item:divsdu} Let $\f \in C^{\infty}_c (\Rn,\Rn)$. The nonlocal divergence is defined, for $x \in \O$, as
			\begin{equation*}
				\diver_{\delta}^s \f(x):= -c_{n,s}\pv_x  \int_{B(x, \d)} \frac{\f(x)+\f(y)}{|x-y|}\cdot\frac{x-y}{|x-y|}\frac{w_\d(x-y)}{|x-y|^{n+s-1}} \, dy .
			\end{equation*}
			\item \label{item:curlsdu} Let $\f \in C^{\infty}_c (\R^3,\R^3)$. Then, for every $x \in \O$ the nonlocal curl, $\curl_\delta^s \f$, is defined as
			\begin{equation} \label{eq: definition of nonlocal curl}
				\curl_\delta^s \f(x):= c_{3,s} \int_{B(x,\delta)} \frac{\f(x)-\f(y)}{|x-y|}\times \frac{x-y}{|x-y|}\frac{w_\d(x-y)}{|x-y|^{2+s}} \, dy.
			\end{equation}
			\item \label{item:NonlocalLaplacian u} Let $u\in C_c^{\infty} (\Rn)$. Then, for every $x \in \O$ the nonlocal Laplacian $\Delta_\delta^s u$ is defined as
			\begin{equation*} \label{eq: definition of nonlocal laplacian}
				\Delta_\delta^s u(x):= \diver^s_\d(D^s_\d u(x)).
			\end{equation*}
		\end{enumerate}
	\end{definition}

Here, we recall that the Cauchy principal value of a singular integral is
\[
    \pv_x\int_{B(x,\delta)}\frac{g(y)}{|x-y|}\frac{x-y}{|x-y|}
        \frac{w_\delta(x-y)}{|x-y|^{n+s-1}}dy
    =\lim_{\e\to0^+}
        \int_{B(x,\delta)\setminus B(x,\e)}\frac{g(y)}{|x-y|}\frac{x-y}{|x-y|}
        \frac{w_\delta(x-y)}{|x-y|^{n+s-1}}dy.
\]

	\begin{rem} \begin{enumerate}[a)]
	    \item  Although not within the scope of this paper, as it was hinted in \cite[Section 6.5.1]{Cue21} and will be addressed in future works, there are different characterizations of the notion of nonlocal Laplacian given in Definition \ref{def: nonlocal gradient} \emph{\ref{item:NonlocalLaplacian u})}, such as, for example, by using the Fourier transform, single singular integrals, or convolution characterizations.
 \item We could have considered the following alternative notation for the nonlocal Laplacian, $-(-\Delta)^s_\d u$, in accordance with the one typically used for the fractional Laplacian $-(-\Delta)^s u$. However, since we have not considered the spectral characterization of this operator, together with the fact that the nonlocal gradient and divergence operators are essential ingredients in this framework, allow us to provide a definition/notation that more closely mimics the classical Laplacian. By eliminating  the $``-"$ in the definition, the presentation also benefits from a simpler notation. 
 \end{enumerate}
	\end{rem}
	Next, we present the definition for the functional spaces, introduced in \cite{BeCuMC22} and which is described mirroring the classical case, where the Sobolev spaces $W^{1,p}(\O)$ are characterized by the closure of $C^\infty_c(\Rn)$ under the natural norm $\|u\|_{W^{1,p}(\O)}=\left(\|u\|_{L^p(\O)}^p+\|\nabla u\|_{L^p(\O)}^p\right)^{\frac{1}{p}}$.
	\begin{definition} \label{de: NL functional spaces density}
		Given $1\leq p< \infty$, the functional space $H^{s,p,\d}(\O)$ is defined as the closure of $C^{\infty}_c(\Rn)$ under the norm
		\begin{equation*}
			\|u\|_{H^{s,p,\d}(\O)}=\left(\|u\|_{L^p(\O_\d)}^p+\|D^s_\d u\|_{L^p(\O)}^p\right)^{\frac{1}{p}}.
		\end{equation*}
	\end{definition}
	
An important feature of the Riesz fractional gradient is that, for smooth functions, it can be expressed as the convolution of the classical gradient with the Riesz potential\cite{ShS2015,COMI2019}. A similar property holds for the nonlocal gradient in \eqref{eq: definition of nonlocal gradient}, and facilitates many of our arguments. To take advantage of this structure, we recall next definition from \cite[Definition 4.1]{BeCuMC22}.
	\begin{definition}\label{de:q}
		Let $0<s<1$ and $\d>0$.
		Define
		\[
		\bar{q}_\d : [0, \infty) \to \R, \qquad q_\d : \Rn \to \R \quad \text{and} \quad Q_\d^s : \Rn \setminus \{0\} \to \R
		\]
		as
		\begin{equation*}
			\bar{q}_\d (t)=(n+s-1)t^{n+s-1} \int_{t}^{\d}\frac{\bar{w}_\d(r)}{r^{n+s}} \, dr , \quad q_{\d} (x) = \bar{q}_\d (|x|) 
   \quad\text{and}\quad
   Q_\d^s(x)=\frac{ q_\d(x)}{\gamma(1-s)|x|^{n+s-1}}.
   		\end{equation*}
	\end{definition}
\noindent
The gradient of $Q_\d^s$ coincides with the integration kernel of Definition \ref{def: nonlocal gradient} (see \cite[Lem. 4.2]{BeCuMC22}). While $Q_\d^s$ has the same order of singularity as the Riesz potential, it is compactly supported and thus integrable, in contrast to the local integrability of the Riesz potential.

	
	\begin{lem} \label{lem: kernel primitive}
		Let $0<s<1$ and $\d>0$.
		Then $Q_\d^s\in L^1 (\Rn)$ is radial, decreasing, and
			
			
			\begin{equation} \label{eq: kernel primitive}
				\frac{-1}{n+s-1}\nabla Q_\d^s(x)=\frac{\rho_\delta{(x)}}{|x|}\frac{x}{|x|}.
			\end{equation}
			
	\end{lem}
\noindent
Next we show that the nonlocal gradient can be written as the gradient of a function convolved with $Q^s_\d$ (\cite[Prop. 4.3]{BeCuMC22}, extended to $H^{s,p,\d}$ later on \cite[Th. 2.13]{CuKrSc23} and \cite[Lem, 4.2]{BeCuMC22b}). This identity has the potential to be useful in theoretical, as well as numerical studies when a solution is known to possess additional smoothness properties. Notice that as the compact support of $Q_\d^s$ is in $B(0,\d)$, for every $x\in \O$, it is enough to consider $\O_{\d}$ as the domain of integration, which makes it appropriate for nonlocal problems where boundary data is usually unavailable on $\Rn\setminus\O$.
	\begin{prop} \label{Prop: convolution with the classical gradient}
		For every $u \in C_c^{\infty}(\Rn)$ and $x \in \Rn$ we have 
		\begin{equation}\label{eq:Dsd=-n}
			D_\delta^s u (x) = \int_{\Rn} \nabla u (y) \, Q_\d^s(x-y) dy 
		\end{equation}
		and $D_\delta^s u  \in C_c^{\infty} (\Rn)$. We recall that $Q_\d^s$ has compact support in $B(0,\d)$. Moreover, if $u\in H^{s,p,\d}(\O)$ with $1\leq p <\infty$, then $Q^s_\d * u \in W^{1,p}(\O)$ and
  \begin{equation*}
      D^s_\d u=\nabla (Q^s_\d *u) \quad \text{ in } \O.
  \end{equation*}
	\end{prop}
Due to the identity \eqref{eq:Dsd=-n}, where the nonlocal derivative is expressed as the convolution of the classical gradient with an integrable kernel, it is possible to provide a Fourier transform formula for functions defined as zero outside a particular domain, or in the Schwartz space $\mathcal{S}$.
	\begin{cor} \label{cor: Fourier transform of the nonlocal gradient}
		For all $u \in \mathcal{S}$ it yields
		\begin{equation*}
			\widehat{D^s_\d u}(\xi)= 2 \pi i \xi \, \widehat{u}(\xi)\cdot \widehat{Q_\d^s}(\xi).
		\end{equation*}
	\end{cor}
	Finally, we recall a representation formula provided in \cite[Th. 4.5]{BeCuMC22} where a function can be recovered from its nonlocal gradient. The result can be viewed as a nonlocal version of the fundamental theorem of calculus, and whose convenience is corroborated by the fact that it has been employed towards proving several results (e.g. Poincaré-Sobolev inequalities or compact embedding theorems) in the nonlocal nonlinear analysis of the operators from Definition \ref{def: nonlocal gradient}. 
    A version for $s=0$ was provided in \cite[Prop. 2.9]{CuKrSc23}.
	\begin{theo} \label{Th: nonlocal version of FTC}
		Let $0<s<1$ and $0<\d$. Then, there exists a function $V_\d^s \in C^{\infty}(\Rn \backslash \{0\},\Rn)$, unique, such that for every $u \in C^\infty_c(\Rn)$ and $x \in \Rn$
		\begin{equation} \label{eq: nonlocal version of FTC}
			u(x)=\int_{\Rn} D^s_\d u(y) \cdot V_\d^s(x-y)\, dy.
		\end{equation}
		In addition, for every $R>0$ there exists $M>0$ such that
		\begin{equation} \label{eq: bound for V}
			|V_\d^s(x)|\leq \frac{M}{|x|^{n-s}}, \qquad x \in B(0,R)\backslash\{0\}.
		\end{equation}
	\end{theo}
    Thus, $V_\d^s$ behaves like $c_{n,-s}\frac{x}{|x|^{n+1-s}}$ as $|x|\to0$, and like $\frac{x}{|x|^n}$ as $|x|\to\infty$ (see~\cite[Th.~6.3.17~\emph{a)}]{Cue21}).

	\section[Nonlocal vector calculus]{Nonlocal vector calculus}\label{NVC}
       This section is meant to gather several nonlocal vector calculus results, including divergence and integration by parts theorems which, for the sake of clarity, require their own subsection. First, linear identities involving composition of nonlocal operators, necessary for the Helmholtz decomposition, are shown. In fact, such results can be obtained from \cite[Props. 4.2, 4.4 and 4.6]{delia2021connections} since the kernel $\rho_\d$ satisfies the required hypothesis therein and despite been initially stated for every $x\in \Rn$, as they remain valid restricted to $\O_{\d}$. However, we show them here as a straightforward consequence of what can be considered a nonlocal version of the Schwartz theorem, which we recall from \cite[Prop. 6.2.6.]{Cue21}, showing symmetry of second nonlocal derivatives. In particular, next result remains valid for $\d=\infty$, i.e. for the Riesz-fractional gradient.
 %
	\begin{prop} \label{Prop: nonlocal Schwartz theorem}
		Let $0<s<1$, $0<\d$ and $u \in C^{\infty}_c(\Rn)$. Let $i \in \{1, \ldots, n\}$ and
		\begin{equation} \label{eq: nonlocal partial derivatives}
			D^s_{\d,i} u(x):= (n+s-1) \int_{B(x,\d)} \frac{u(x)-u(y)}{|x-y|} \frac{x_i -y_i}{|x-y|} \rho_\d(x-y) \, dy.
		\end{equation}
		Then, for every $i,j \in \{1, \ldots, n\}$, the following equality holds
		\begin{equation*}
			D_{\d,j}^s(D_{\d,i}^s u)=	D_{\d,i}^s(D_{\d,j}^s u).
		\end{equation*}
	\end{prop}
	\begin{proof}
	By making the change of variables $y=\bar{y} +x'-x$ and renaming $\bar{y}=y$, we have that:
		\begin{align*}
			&\frac{D_{\d,i}^s u(x)-D_{\d,i}^s u(x')}{(n+s-1)}\\
			&=\int_{B(x,\d)} \frac{u(x)-u(y)}{|x-y|} \frac{x_i -y_i}{|x-y|} \rho_\d(x-y) \, dy-\int_{B(x',\d)} \frac{u(x')-u(y)}{|x'-y|} \frac{x'_i -y_i}{|x'-y|} \rho_\d(x'-y) \, dy  \\
			&=\int_{B(x,\d)} \frac{[u(x)-u(y)-u(x')+u(y+x'-x)]}{|x-y|^2}(x_i-y_i) \rho_\d(x-y) \, dy.
		\end{align*}
		Therefore, 
		\begin{equation} \label{eq: nl Schwartz theorem 1}
			\begin{split}
			\frac{D_{\d,j}^s(D_{\d,i}^s u)(x)}{(n+s-1)^2}= 
			\int_{B(x,\d)}\int_{B(x,\d)} &\frac{[u(x)-u(y)-u(x')+u(y+x'-x)]}{|x-y|^2|x-x'|^2} \, \cdot \\& \qquad\cdot (x_i-y_i) \rho_\d(x-y) (x_j-x'_j) \rho_\d(x-x') \,dy dx'.
			\end{split}
		\end{equation}
	Applying Fubini's theorem we obtain that
		\begin{align*}
			\frac{D_{\d,j}^s(D_{\d,i}^s u)(x)}{(n+s-1)^2}=&\int_{B(x,\d)} \frac{\int_{B(x,\d)} \frac{u(x)-u(x')}{|x-x'|^2}(x_j-x'_j) \rho_\d(x-x') dx' }{|x-y|^2} (x_i-y_i) \rho_\d(x-y) dy	 \\
			&-\int_{B(x,\d)} \frac{\int_{B(x,\d)} \frac{u(y)-u(y+x'-x)}{|x-x'|^2}(x_j-x'_j) \rho_\d(x-x') dx' }{|x-y|^2} (x_i-y_i) \rho_\d(x-y) dy.	
		\end{align*}
		Making the change of variables $\bar{x}=x'-x+y$, yields
		\begin{multline}  \label{eq: nl Schwartz theorem 2}
				\int_{B(x,\d)} \frac{u(y)-u(y+x'-x)}{|x-x'|^2}(x_j-x'_j) \rho_\d(x-x') dx\\
				=\int_{B(y,\d)} \frac{u(y)-u(\bar{x})}{|y-\bar{x}|^2}(y_j-\bar{x}_j) \rho_\d(y-\bar{x}) d\bar{x}= \frac{D_{\d,j}^su(y)}{(n+s-1)}.
		\end{multline}
		Therefore, combining \eqref{eq: nl Schwartz theorem 1} and \eqref{eq: nl Schwartz theorem 2}, we have that
		\[
			\frac{D_{\d,j}^s(D_{\d,i}^s u)(x)}{(n+s-1)^2}= \int_{B(x,\d)} \frac{D_{\d,j}^s u(x)-D_{\d,j}^su(y)}{|x-y|^2} (x_i-y_i) \rho_\d(x-y)\, dy
	=\frac{D_{\d,i}^s(D_{\d,j}^s u)(x)}{(n+s-1)^2},
		\]
		and the result follows.
	\end{proof}
	\begin{rem} \label{rem: nonlocal Schwartz theorem}
		If we define each component of the nonlocal divergence of a field $\f \in C^{\infty}_c(\Rn,\Rn)$ as
		\begin{equation} \label{eq: nonlocal divergence components}
			\diver^s_{\d,i} \f_i(x)=-(n+s-1)\pv_x  \int_{B(x,\d)} \frac{\f_i(x)+\f_i(y)}{|x-y|} \frac{x_i -y_i}{|x-y|} \rho_\d(x-y) \, dy,
		\end{equation}
		by odd symmetry we have that
		\begin{equation} \label{equality nonlocal gradient and divergence componentwise}
			\diver^s_{\d,i} \f_i(x)=D^s_{\d,i} \f_i(x).
		\end{equation} 
		Then, and as a consequence of Proposition \ref{Prop: nonlocal Schwartz theorem}, we have that for $u \in C^{\infty}_c(\Rn)$
		\begin{equation*}
			\diver_{\d,j}^s(D_{\d,i}^s u)=	\diver_{\d,i}^s(D_{\d,j}^s u) \quad \text{ and } \quad
			\diver_{\d,j}^s(D_{\d,i}^s u)=	D_{\d,j}^s(\diver_{\d,i}^s u).
		\end{equation*}
	\end{rem}
We remark that the previous result as well as the next one involving linear identities could be extended to hold on a functional space where the nonlocal derivatives $D^s_\d$ are well defined, as similarly mentioned in Remark \ref{rem: comment on the generalisation for second NL derivatives}. The previously referred identities are presented below. Their proof, based on Proposition \ref{Prop: nonlocal Schwartz theorem}, share the same spirit as \cite{delia2021connections}.
	\begin{prop} \label{Prop: vectorial linear identities}
		Let $u\in C^{\infty}_c(\R^3)$ and $\f\in C^\infty_c(\R^3,\R^3)$ then the following identities hold.
		\begin{enumerate}[1)]
			\item $ \curl^s_\d D^s_\d u(x)=0$
			\item $\diver^s_\d \curl^s_\d \f(x)=0$
			\item $\curl^s_\d \curl^s_\d \f(x)= D^s_\d \diver^s_\d \f (x) - \diver^s_\d D^s_\d \f (x)$.
		\end{enumerate}
	\end{prop}
	\begin{proof}
		The proof follows as in the classical case using Proposition \ref{Prop: nonlocal Schwartz theorem} and Remark \ref{rem: nonlocal Schwartz theorem} instead of the classical Schwarz theorem for the symmetry of second derivatives.
	\end{proof}
	Even more basic identities are the following, yet they are still needed to clarify several steps taken throughout this analysis. They involve reflection, $\widetilde{f}(x):= f(-x)$, and translation, $f_y=f(x-y)$, $x,y \in \O$, operations for nonlocal derivatives.
	\begin{prop} \label{prop: basic identities}
		Let $0\leq s<1$, $1\leq p <\infty$, $\f \in C^\infty_c(\Rn,\Rn)$ and $f 
   \in H^{s,p,\delta}(\O)$. Then, the following equalities hold almost everywhere
		\begin{enumerate}[a)]
			\item $(
			D^s_\d f)(-x)=-(D^s_\d \widetilde{f})(x)$, \label{prop: basic identity a}
			\item $\diver_{\delta}^s\big(\f(y) \, f(x-y)\big)(x)=\f(y) \cdot D^s_\d f_y(x)$, \label{prop: basic identity b}
			\item $D_\d^s f_y (x)=(D^s_\d f)(x-y)=-(D_\d^s \widetilde{f})(y-x)=-(D_\d^s \widetilde{f}_x )(y)$, \label{prop: basic identity c}
                \item If $n=3$, $\curl_{\delta}^s\big(\f(y) \, f(x-y)\big)(x)=\f(y) \times D^s_\d f_y(x)$. \label{prop: basic identity d}
		\end{enumerate}
	\end{prop}
	\begin{proof}
	We prove the result for smooth functions as a standard argument using integration against test functions would render the generalization for functions in $H^{s,p,\d}(\O)$ almost everywhere. 
 We start with  \emph{\ref{prop: basic identity a})}, which is trivial. By odd symmetry and a change of variables we have
		\begin{equation*}
			\begin{split}
				D^s_\d f(-x)&=c_{n,s}\pv_x \int_{B(-x,\d) }\frac{-f(z)}{|-x-z|} \frac{-x-z}{|-x-z|}\frac{w_\d(-x-z)}{|-x-z|^{n+s-1}}\, dz \\
				&=-c_{n,s} \pv_x\int_{B(x,\d)} \frac{-f(-\bar{z})}{|x-\bar{z}|} \frac{(x-\bar{z})}{|x-\bar{z}|}\frac{w_\d(x-\bar{z})}{|x-\bar{z}|^{n+s-1}}\, d\bar{z}=-D^s_\d \widetilde{f}(x).
			\end{split}
		\end{equation*}
		Regarding \emph{\ref{prop: basic identity b})}, we observe that we may factor the term $\f(y)$ out of the integral in the definition of the nonlocal divergence: 
		\begin{align*}
			\diver_{\delta}^s\big(\f(y) \, f(x-y)\big)(x)&=- c_{n,s}\pv_x\int_{B(x,\d)} \frac{ \f(y) f(x-y) -\f(y) f(z-y) }{|x-z|} \cdot \frac{x-z}{|x-z|}\frac{w_\d(x-z)}{|x-z|^{n+s-1}}\, dz \\
			&=\f(y)\cdot D_\d^s f_y(x).
		\end{align*}
		Finally, for \emph{\ref{prop: basic identity c})}, we make the change of variables $\bar{z}=z-y$ to write
		\begin{align*}
			D^s_\d f_y(x)&=c_{n,s}\int_{B(x,\d)} \frac{f(x-y)-f(z-y)}{|x-z|}\frac{x-z}{|x-z|} \frac{w_\d(x-z)}{|x-z|^{n+s-1}}\, dz \\
			&=c_{n,s}\int_{B(x-y,\d)} \frac{f(x-y)-f(\bar{z})}{|x-\bar{z}-y|}\frac{x-\bar{z}-y}{|x-\bar{z}-y|} \frac{w_\d(x-\bar{z}-y)}{|x-\bar{z}-y|^{n+s-1}}\, d\bar{z} =(D^s_\d f)(x-y).
		\end{align*}
		Making another change of variables $\bar{z}=-z$ yields
		\begin{align*}
			(D^s_\d f)(x-y)=\int_{B(y-x,\d) } \frac{(f(x-y)-f(-z))}{|x+z-y|} \frac{x-y+z}{|(y-x)-x|}\frac{w_\d((y-x)-z)}{|(y-x)-z|^{n+s-1}}\, dz=-D^s_\d\widetilde{f}(y-x).
		\end{align*}
		The last equality in \emph{\ref{prop: basic identity c})} is obtained as the first.
            Finally, the proof of \emph{\ref{prop: basic identity d})} is identical to that of \emph{\ref{prop: basic identity b})}.
	\end{proof}
We conclude this first part of Section \ref{NVC} by recalling the nonlocal derivative of a product, showing a clear difference with respect to the classical case. For doing so we need next operator defined and studied in \cite[Def. 3.2 and Lem. 3.3]{BeCuMC22b}.
	\begin{lem} \label{Lema operador lineal delta} 
		Let $1\leq q< \infty$, $0<\d$ and $0 < s< 1$. Then, for each $\varphi\in C^{0,1}(\O_\d)$, $1\le r\le q$, and $k\in\N$, the operator $K^{s,\d}_{\varphi}: L^r(\O_\d, \R^{n\times k}) \rightarrow L^q(\O,\R^k)$ defined as
		\begin{equation*}
			K^{s,\d}_{\varphi}(U)(x):= c_{n,s} \int_{B(x,\d)} \frac{\varphi(x)-\varphi(y)}{|x-y|^{n+s}} U(y)\cdot \frac{x-y}{|x-y|}w_\d(x-y) dy , \qquad \text{a.e. } x \in \O ,
		\end{equation*}
		is linear and bounded; i.e. there exists a constant $C_1=C_1(n,s,q,\d,r,w_\d,|\O|)>0$ such that
		\begin{equation*}
			\|K_{\varphi}^s(U) \|_{L^q(\O,\R^k)} \leq [\varphi]_{C^{0,1}(\O_\d)}^qC_1 \| U\|_{L^r(\O_\d,\R^{n \times k})}.
		\end{equation*}
		Here, $[\varphi]_{C^{0,1}(\O_\d)}$ denotes the Lipschitz semi-norm of $\varphi$. In addition, analogously, when $n=3$, the same result holds for $K^{s,\d}_{\f}(\cdot\times): L^r(\O_\d, \Rn) \rightarrow L^q(\O, \mathbb{R}^n)$ defined with a cross product instead of a dot product.
  \begin{equation*}
			K^{s,\d}_{\varphi}(U\times)(x):= c_{3,s} \int_{B(x,\d)} \frac{\varphi(x)-\varphi(y)}{|x-y|^{3+s}} U(y)\times \frac{x-y}{|x-y|}w_\d(x-y) dy , \qquad \text{a.e. } x \in \O .
		\end{equation*}

	\end{lem}
	The nonlocal Leibniz formula (see \cite[Lem. 3.4]{BeCuMC22b} and \cite[Lem. 3.2]{BeCuMC}) is introduced next. Notice that the difference with its local analogous lies with the lineal operator $K^{s,\d}$, substituting one of the products of the classical case, as the interactions given by the product take place within the integral defining such term.
	\begin{lem} \label{lem: Divergencia no local producto}
		Let $0 < s < 1$, $0<\d$ and $1 \leq p < \infty$.
		Let $g \in H^{s,p,\d}(\O,\mathbb{R}^n)$ and $\f \in C^1(\overline{\O_{\d}})$.
		Then $\f g \in H^{s,p,\d}(\O,\mathbb{R}^n)$ and for a.e.\ $x \in \mathbb{R}^n$,
		\begin{equation} \label{eq: NL Leibniz rule div}
			\diver_\d^s (\f g) (x) = \f(x) \diver_\d^s g(x)+ K_{\f}^{s,\d} (g) (x) .
		\end{equation}
  In addition, when $n=3$ we also have
  \begin{equation} \label{eq: NL Leibniz rule curl}
			\curl_\d^s (\f g) (x) = \f(x) \curl_\d^s g(x)+ K_{\f}^{s,\d} (g \times) (x) .
		\end{equation}
	\end{lem}
 \begin{proof}
     \eqref{eq: NL Leibniz rule div} is proven in \cite[Lem. 3.5]{BeCuMC22b}. \eqref{eq: NL Leibniz rule curl} is analogous. First we assume $g\in C^\infty_c(\R^3)$ and use that $\f$ can be extended to a function in $C^\infty_c(\R^3)$. Thus,
     \begin{equation}
    \begin{split}
         \curl_\delta^s \f g(x)&= c_{3,s} \int_{B(x,\delta)} \frac{\f(x) g(x) - \f(x) g(y) + \f(x) g(y)-\f(y) g(y)}{|x-y|}\times \frac{x-y}{|x-y|}\frac{w_\d(x-y)}{|x-y|^{2+s}} \, dy\\
         &=\f(x) \curl^s_\d g(x) + K^{s,\d}_\f (g\times).
    \end{split}
     \end{equation}
     The result follows from extension via density, as in \cite[Lem. 3.4]{BeCuMC22b}.
 \end{proof}


\subsection[Nonlocalised normal direction]{Nonlocalised normal direction} \label{subsec: technical results}

In this section, we introduce a nonlocal boundary operator that allows for a concise integration by parts formula. 


\begin{definition} \label{def: NL normal direction}
Let $0\leq s < 1$. We define the following term whenever it makes sense.
\begin{enumerate}[a)]
    \item For $u:\O_\d \to \R$ a measurable function,
    \begin{equation*}		\mathcal{N}u(x):=c_{n,s}\pv_x\int_{\O} \frac{u(y)}{|x-y|} \frac{x-y}{|x-y|} \frac{w_\delta(x-y)}{|x-y|^{n-1+s}}\in\, dy\in\Rn, \qquad x \in \Gamma_{\pm\d}.
	\end{equation*}
 \item For $\phi:\O_\d \to \Rn$ a measurable function,
    \begin{equation*}		\mathcal{N}\phi(x):=c_{n,s}\pv_x\int_{\O} \frac{\phi(y)}{|x-y|}\cdot \frac{x-y}{|x-y|} \frac{w_\delta(x-y)}{|x-y|^{n-1+s}}\, dy\in\R, \qquad x \in \Gamma_{\pm\d}.
	\end{equation*}
 \end{enumerate}
\end{definition}
\begin{rem}
\begin{enumerate}[(a)]
    \item We note that if $x\in\Gamma_\d$, then $\mathcal{N}u(x)$ and $\mathcal{N}\f(x)$ are well-defined for any integrable $u,\f$. The principal value is only needed for $x\in\Gamma_\d$.
    \item We can identify $\nu_\O(x)=\mathcal{N}\chi_\O(x)$ as a nonlocal normal vector at $x\in\Gamma_{\pm\d}$. An analogue of this, and its connection to nonlocal curvature for general measurable sets, was studied in~\cite{Rossi-Curvature}. 
\end{enumerate}
    
\end{rem}
In the next lemma, we identify a space of functions on which $\mathcal{N}$ is a bounded operator.
\begin{lem} \label{lem: controlling the boundary term}
	Let $0\leq s <1$, $1\leq p_1\leq \infty$ and $\e>0$ be given. Then, for every $p_2>\frac{1}{1-s}$ and $q< \frac{1}{s}$, we have the following.
	\begin{enumerate}[a)]
		\item $\mathcal{N}:L^{p_1}(\O) \cap L^{p_2}(\Gamma_{\pm\e}) \to L^{p_1}(\Gamma_\d\setminus\Gamma_\e)$ with
		\[
		\|\mathcal{N}u\|_{L^{p_1}(\Gamma_\d\setminus\Gamma_\e)}\leq \frac{c_{n,s}a_0}{\e^{(n+s)}} |\O|^{\frac{1}{p_1'}} |\Gamma_\d|^{\frac{1}{p_1}}\|u\|_{L^{p_1}(\O)}.
		\]
		\item $\mathcal{N}:L^{p_1}(\O) \cap L^{p_2}(\Gamma_{\pm\e}) \to L^q(\Gamma_\e,\Rn)$ with
		\[
		\|\mathcal{N}u\|_{L^q(\Gamma_{\e},\Rn)}\leq 	c_{n,s}M_{s,q}(\Gamma_{\d})^{1-\theta}	M_{s,p_2'}(\O)^{\theta}\|u\|_{L^{p_2}(\Gamma_{\pm\e})}
        + \frac{c_{n,s}a_0}{\e^{(n+s)}} |\O|^{\frac{1}{p_1'}} |\Gamma_\d|^{\frac{1}{p_1}}\|u\|_{L^{p_1}(\O)} .
		\]
	\end{enumerate}
where $\theta \in (0,1)$ and, for $O\subseteq \Rn$,
\begin{equation}\label{Eq:M_sq}
	M_{s,q}(O)=  \sigma_{n-1} a_0\begin{cases}
		\frac{1}{s} \left(\int_{O} \left[\frac{1}{\dist(x,\partial \O)^s}\right]^q\,dx\right)^{\frac{1}{q}}, & \text{ if }0<s<1\\
		\left( \int_{O}\left|\log(\dist(x,\partial \O)) \right|^q\, dx\right)^{\frac{1}{q}},& \text{ if }s=0.
	\end{cases} 
\end{equation}
Here $\sigma_{n-1}$ is the surface area of the unit sphere in $\Rn$.
\end{lem}
\begin{proof}
	\emph{Part a).} Let us first address the integrability on $\Gamma_\d\setminus\Gamma_\e$. By H\"older's inequality and the assumption that $\|w_\d\|_{L^\infty(\Rn)}=a_0$,
	\begin{equation*}
	\begin{split}
			&\left(\int_{\Gamma_\d\setminus\Gamma_\e} \left|\int_{\O}\frac{u(y)}{|x-y|}\frac{x-y}{|x-y|} \frac{w_\delta(x-y)}{|x-y|^{n-1+s}}\, dy\right|^{p_1}\,dx\right)^{\frac{1}{p_1}}\\ &\qqquad\leq
			\left(\int_{\Gamma_\d\setminus\Gamma_\e} \|u\|_{L^p(\O)}^{p_1} \left(\int_{\O} \frac{a_0^{p_1'}}{|x-y|^{(n+s)p_1'}}\,dy\right)^{\frac{p_1}{p_1'}}\, dx\right)^{\frac{1}{p_1}}\leq \|u\|_{L^{p_1}(\O)} \frac{a_0}{\e^{(n+s)}} |\O|^{\frac{1}{p_1'}} |\Gamma_\d|^{\frac{1}{p_1}}.
	\end{split}
	\end{equation*}
The conclusion follows.

	\emph{Part b).} Let $x\in \Gamma_\d$, we split
	\begin{equation*}
		c_{n,s}^{-1}\mathcal{N}u(x)=\int_{\O\cap B(x,\e)^c}\frac{u(y)}{|x-y|}\frac{x-y}{|x-y|} \frac{w_\delta(x-y)}{|x-y|^{n-1+s}}\, dy+\int_{\O\cap B(x,\e)}\frac{u(y)}{|x-y|}\frac{x-y}{|x-y|} \frac{w_\delta(x-y)}{|x-y|^{n-1+s}}\, dy.
	\end{equation*}
	 As in\emph{Part a)}, we bound the first term by
	\begin{equation*}
		\int_{\O_\d }\left|\int_{\O\cap B(x,\e)^c}\frac{u(y)}{|x-y|}\frac{x-y}{|x-y|} \frac{w_\delta(x-y)}{|x-y|^{n-1+s}}\, dy\right|^{p_1} \, dx \leq \frac{a_0^{p_1}}{\e^{(n+s)p_1}} |\O|^{\frac{p_1}{p_1'}} |\Gamma_\d\|u\|_{L^{p_1}(\O)}^{p_1} .
	\end{equation*}
	Turning to the integrability
	\begin{equation*}
		\mathcal{N}^\e u(x):=c_{n,s}\int_{\O\cap B(x,\e)}\frac{u(y)}{|x-y|}\frac{x-y}{|x-y|} \frac{w_\delta(x-y)}{|x-y|^{n-1+s}}\, dy,\qquad x \in \Gamma_{\pm\d},
	\end{equation*}
	we will use the Riesz-Thorin interpolation theorem. The proof is divided into four parts. First, we study the asymptotic behaviour of
	\begin{equation} \label{eq: boundary term integrable}
		\int_{\O\cap B(x,\e)} \frac{1}{|x-y|}\frac{x-y}{|x-y|}\frac{w_\d(x-y)}{|x-y|^{n-1+s}} \, dy 
	\end{equation}
	as $x$ approaches $\partial \O$.
	In the second and third parts, we will see that the statement holds for functions in $L^\infty$ and $L^{p_2}$, respectively. For the last part, we apply Riesz-Thorin interpolation theorem to conclude the result.\\	
	\emph{Part 1.} Fix $x\in\Gamma_\d$. 
	As $|w_\d|\leq a_0$, the co-area formula implies 
	\begin{equation}\label{eq: bound of the boundary term}
		\begin{split}
		\int_{\O\cap B(x,\e)}\frac{|w_\d(x-y)|}{|x-y|^{n+s}} \, dy
        &\leq \int_{\dist(x,\p \O)}^{\e} \frac{a_0}{r^{n+s}} \int_{\partial B(x,r) \cap \Omega} \, d \mathcal{H}^{n-1}(z) \, dr \\
	&\leq  \int_{\dist(x,\p \O)}^{\e} \frac{a_0}{r^{n+s}}\sigma_{n-1} r^{n-1}  \, dr\\
        &=\sigma_{n-1}a_0
        \left\{\begin{array}{ll}
           \frac{1}{s} \left[\frac{1}{\dist(x,\partial \O)^s}-\frac{1}{\e^s}\right],  & \text{ if }0<s<1, \\
            \log(\e)-\log(\dist(x,\partial \O)), & \text{ if }s=0.
        \end{array}\right.
		\end{split}
	\end{equation}
	Thus, the term in \eqref{eq: boundary term integrable} is in $L^{\bar{q}}(\Gamma_\d)$ for every $\bar{q}<\frac{1}{s}$.\\	
	\emph{Part 2.} Fix $u \in L^{\infty}(\Gamma_{\pm\e})$. Then,
	\begin{equation} \label{eq: controlling boundary term infinty case 1}
		\left|\int_{\O\cap B(x,\e)} \frac{u(y)}{|x-y|}\frac{x-y}{|x-y|}\frac{w_\d(x-y)}{|x-y|^{n-1+s}} \, dy   \right|\leq \int_{\O\cap B(x,\e)} \frac{\|u\|_{L^\infty(\Gamma_{\pm\e})}}{|x-y|}\frac{|w_\d(x-y)|}{|x-y|^{n-1+s}} \, dy .
	\end{equation}
	Applying \eqref{eq: bound of the boundary term} and \eqref{eq: controlling boundary term infinty case 1},
	\begin{equation*} 
			\|\mathcal{N}^\e u\|_{L^{q}(\Gamma_\d)}\leq c_{n,s}M_{s,q}(\Gamma_\delta)	\|u\|_{L^\infty(\Gamma_{\pm\e})},
	\end{equation*}
which is finite for every $q<\frac{1}{s}$.\\
	\emph{Part 3.} Now, let $u \in L^{p_2}(\Gamma_{\pm\e})$, with ${p_2}>\frac{1}{1-s}$, be given. By Minkowski integral inequality, we have
	\begin{equation} \label{eq: bound of the boundary term Part 3}
		\int_{\Gamma_\d} \left| \int_{\O\cap B(x,\e)} \frac{u(y)}{|x-y|}\frac{x-y}{|x-y|} \frac{w_\delta(x-y)}{|x-y|^{n-1+s}}\, dy \right| \, dx \leq \int_{\O} \int_{\Gamma_{\d}} |u(y)| \frac{w_\delta(0) \chi_{B(y,\e)}}{|x-y|^{n+s}}\, dx\, dy.
	\end{equation}
	Now, as \eqref{eq: bound of the boundary term} remains true with $\O$ and $\Gamma_\d$, and $y$ and $x$, interchanged, respectively, we find, for $y \in \O$, that
	\begin{equation*}
		\int_{\Gamma_{\d}} \frac{w_\delta(0) \chi_{|x-y|<\e}}{|x-y|^{n+s}} \, dx  \leq \sigma_{n-1} a_0\begin{cases}
			\frac{1}{s} \left[\frac{1}{\dist(y,\partial \O)^s}-\frac{1}{\e^s}\right], & \text{ if } 0<s<1,\\
			\log(\e)-\log(\dist(y,\partial \O)),& \text{ if }s=0.
		\end{cases}
	\end{equation*}
Then, recalling~\eqref{Eq:M_sq},
\begin{equation*}
	\left(\int_{\O}\left|\int_{\Gamma_\d}  \frac{w_\delta(0) \chi_{B(y,\e)}}{|x-y|^{n+s}}\, dx\right|^{p_2'}\right)^{\frac{1}{p_2'}}\leq M_{s,p_2'}(\O)
\end{equation*}
	Notice that the integration over $\O$ in \eqref{eq: bound of the boundary term Part 3} is equivalent to the integration over $\O \cap \Gamma_{\pm\e}$.  Applying H\"older inequality we get
	\begin{equation*}
		\int_{\O \cap \Gamma_{\pm\e}} |u(y)|\int_{\Gamma_\d}  \frac{w_\delta(0) \,  \chi_{B(y,\e)}}{|x-y|^{n+s}}\, dx\, dy \leq M_{s,p_2'}(\O)\|u\|_{L^{p_2}(\Gamma_{\pm\e})}
	\end{equation*}
	which is finite since for ${p_2}>\frac{1}{1-s}$ we have that $sp_2'=\frac{sp_2}{p_2-1}<1$.\\
	\emph{Part 4.} We already have $\mathcal{N}:L^{p_2}(\Gamma_{\pm\e}) \to L^1(\Gamma_{\d})$ and $\mathcal{N}^\e:L^\infty(\Gamma_{\pm\e}) \to L^q(\Gamma_\d)$ are bounded linear operators, for every $q<\frac{1}{s}$. By Riesz-Thorin interpolation theorem, we conclude that $\mathcal{N}^\e:L^{p_2}(\Gamma_{\pm\e}) \to L^q(\Gamma_\d)$	satisfies
	\begin{equation*}
		\|\mathcal{N}^\e u\|_{L^q(\Gamma_{\d})}\leq c_{n,s}M_{s,q}(\Gamma_\d)^{1-\theta} M_{s,p_2'}(\O)^\theta \|u\|_{L^{p_2}(\Gamma_{\pm\e})},
  \quad\text{ for all $0<\theta<1$, ${p_2}>\textstyle{\frac{1}{1-s}}$, and $q<\textstyle{\frac{1}{s}}$.}
	\end{equation*}
	 Indeed, due to \emph{Part 2} and \emph{Part 3} and the Riesz-Thorin interpolation theorem, we see that $\mathcal{N}^\e:L^{p_\theta}(\Gamma_{\pm\e}) \to L^{q_\theta}(\Gamma_{\d})$ is a bounded linear operator and
	\begin{equation*}
		\frac{1}{p_\theta}=\frac{1-\theta}{{p_2}}+0 \qquad \frac{1}{q_\theta}=1-\theta +\frac{\theta}{q},\quad\text{ with $0<\theta<1$}.
	\end{equation*}
	Since $\frac{1}{q}>s$, 
	\begin{equation*}
		\frac{1}{q_\theta}=1-\theta +\frac{\theta}{q}>1-\theta +\theta s>s \quad \Longleftrightarrow \quad \theta>1.
	\end{equation*}
	Thus ${p_2}>\frac{1}{1-s}$ and $p_{\theta}>\frac{1}{1-s} \frac{1}{(1-\theta)}>\frac{1}{1-s}$, for any $0<\theta<1$. The lemma is established.
\end{proof}


\subsection[Divergence Theorem]{Divergence Theorem}
In \cite[Th. 3.2]{BeCuMC22} a nonlocal integration by parts for smooth functions where at least one has compact support was shown. Here we are interested in situations where neither function is required to be compactly supported. Within this framework, we obtain a candidate for nonlocal Neumann boundary conditions, which was partially studied in subsection \ref{subsec: technical results}. To this end, we establish nonlocal divergence theorem that includes a double volumetric one boundary boundary term.

Recall that $\Gamma_{\pm\e}:=\{x \in \O_\d : \dist(x,\partial \O) < \e \}$, for $0<\e<\delta$. In the following statement, notice that the vector $\frac{x-y}{|x-y|}$ points from the outer collar into $\O$.
\begin{theo}[\textbf{Divergence theorem}] \label{th: nonlocal divergence theorem}
	Let $0\leq s<1$, $1<p<\infty$ and $\e>0$. Assume $\f\in H^{s,p,\d}(\O,\Rn)\cap L^q(\Gamma_{\pm\e},\Rn$) for some $q>\frac{1}{1-s}$, then the following equality holds 
	\begin{equation*}
		\int_{\O} \diver^s_\d \f(x) \, dx=  -c_{n,s}\int_{\Gamma_{\d}}\int_{\O} \frac{\f(x)+\f(y)}{|x-y|}\cdot\frac{x-y}{|x-y|}\frac{w_\d(x-y)}{|x-y|^{n+s-1}} \, dx \,dy .
	\end{equation*}
\end{theo}
\begin{proof}
	%
	\emph{Part 1.} We prove it first for smooth functions $\f\in C^{\infty}_c(\Rn)$. By definition \ref{def: nonlocal gradient},
	\begin{align*}
		\frac{1}{c_{n,s}}\int_{\O} \diver^s_\d \f(x) \, dx&=- \int_{\O} \pv_x \int_{\O_\d} \frac{\f(x)+\f(y)}{|x-y|}\cdot\frac{x-y}{|x-y|}\frac{w_\d(x-y)}{|x-y|^{n+s-1}} \, dy \,dx \\
		&= -\lim_{\e \to 0}\int_{\O_\d} \int_{\O} \frac{\f(x)+\f(y)}{|x-y|}\cdot\frac{x-y}{|x-y|}\frac{w_\d(x-y)}{|x-y|^{n+s-1}} \chi_{B(y,\e)^c} \, dx \,dy ,
	\end{align*}
	where the limit can be taken out of the integral thanks to the odd symmetry. We may thus rewrite the inner integral in the second term as 
	\[\left|\int_{B(x, \d)}\frac{\f(x)+\f(y)}{|x-y|}\cdot\frac{x-y}{|x-y|}\frac{w_\d(x-y)}{|x-y|^{n+s-1}}  \chi_{B(y,\e)^c} \, dy\right| \leq \|\nabla u\|_{\infty} (n+s-1) \|\rho_\d\|_{L^1(\Rn)}.\]
	Then, we split the integral in two and notice that one of the terms is zero by symmetry
	\begin{align*}
		\frac{1}{c_{n,s}} \int_{\O} \diver^s_\d \f(x) \, dx=& -\lim_{\e \to 0}\int_{\O}  \int_{\O} \frac{\f(x)+\f(y)}{|x-y|}\cdot\frac{x-y}{|x-y|}\frac{w_\d(x-y)}{|x-y|^{n+s-1}} \chi_{B(y,\e)^c} \, dx \,dy  \\
		& - \lim_{\e \to 0}\int_{\Gamma_{\d}}  \int_{\O} \frac{\f(x)+\f(y)}{|x-y|}\cdot\frac{x-y}{|x-y|}\frac{w_\d(x-y)}{|x-y|^{n+s-1}} \chi_{B(y,\e)^c} \, dx \,dy  \\
		=&-\lim_{\e \to 0}\int_{\Gamma_{\d}}  \int_{\O} \frac{\f(x)+\f(y)}{|x-y|}\cdot\frac{x-y}{|x-y|}\frac{w_\d(x-y)}{|x-y|^{n+s-1}} \chi_{B(y,\e)^c} \, dx \,dy  .
	\end{align*}
	Now, we find 
	\[
	(x,y)\mapsto\frac{\f(x)+\f(y)}{|x-y|}\cdot\frac{x-y}{|x-y|}\frac{w_\d(x-y)}{|x-y|^{n+s-1}} \in L^1(\O \times\Gamma_\d).
	\]
	Hence, we can transfer the last limit with the integral, and therefore,
	\begin{equation*}
		\int_{\O} \diver^s_\d \f(x) \, dx= -c_{n,s}\int_{\Gamma_{\d}} \int_{\O} \frac{\f(x)+\f(y)}{|x-y|}\cdot\frac{x-y}{|x-y|}\frac{w_\d(x-y)}{|x-y|^{n+s-1}} \, dx \,dy .
	\end{equation*}
	\emph{Part 2.} Now, let $\f \in H^{s,p,\d}(\O,\Rn)\cap L^q(\Gamma_{\pm\e},\Rn)$ and $0<r<\e$, then, by Proposition \eqref{prop sequence with mollifiers} we can take a sequence $\{\f_j\}_{j\in \N}\subset C^{\infty}_c(\Rn)$ such that
	\begin{equation*}
		\f_j \to \f \text{ in } L^q(\Gamma^r,\Rn), \,\, \f_j \to \f \text{ in } L^p(\O_{\d-r},\Rn), \quad \text{ and } \quad D^s_\d \f_j \to D^s_\d \f \text{ in }  L^p(\O_{-r}, \Rnn).
	\end{equation*}
	Next we apply the nonlocal divergence theorem for the smooth functions $\f_j$ on the domain $\O_{-r}$.
	\begin{equation} \label{eq: nl div Th for smooth functions}
		\int_{\O_{-r}} \diver^s_\d \f_j(x) \, dx
    =-c_{n,s}\int_{\O_{\d-r}\backslash \O_{-r}} \int_{\O_{-r}} \frac{\f_j(x)+\f_j(y)}{|x-y|}\cdot\frac{x-y}{|x-y|}\frac{w_\d(x-y)}{|x-y|^{n+s-1}} \, dx \,dy .
	\end{equation}
	Finally, by Lemma \ref{lem: controlling the boundary term} we have that there exists $C>0$ such that
	\begin{equation*}
		\begin{split}
			\int_{\O_{\d-r}\backslash \O_{-r}}  \int_{\O_{-r}}\left| \frac{\f(x)}{|x-y|}\frac{w_\d(x-y)}{|x-y|^{n+s-1}}\right| \, dx \,dy &\leq C \|\f\|_{L^p(\O_{\d-r})}\\
			  \int_{\O_{-r}}	\int_{\O_{\d-r}\backslash \O_{-r}} \left| \frac{\f(y)}{|x-y|}\frac{w_\d(x-y)}{|x-y|^{n+s-1}}\right| \, dy \, dx &\leq C \|\f\|_{L^p(\O_{\d-r})},
		\end{split}
	\end{equation*}
	which, combined with \eqref{eq: nl div Th for smooth functions} and $\f_j \to \f$ in $H^{s,p,\d}(\O)$ imply 
	\begin{equation*}
		\int_{\O_{-r}} \diver^s_\d \f(x) \, dx
    =  -c_{n,s}\int_{\O_{\d-r}\backslash \O_{-r}}\int_{\O_{-r}} \frac{\f(x)+\f(y)}{|x-y|}\cdot\frac{x-y}{|x-y|}\frac{w_\d(x-y)}{|x-y|^{n+s-1}} \, dx \,dy .
	\end{equation*}
	Now, the results follow taking $r$ to zero, as the constants appearing in the inequalities in Lemma \ref{lem: controlling the boundary term}, i.e.$M_{s,q}(\O_{\d-r}\backslash \O_{-r}),	M_{s,p_2'}(\O_{-r}),|\O_{-r}|$, and $ |\O_{\d-r}\backslash\O_{-r}|$ can be bounded, respectively, by $M_{s,q}(\Gamma_{\d})$,	$M_{s,p_2'}(\O)$, $|\O|$, and  $|\O_{\d}\backslash\O|$. Indeed, it is straightforward that $\left| \O_{-r}\right|\leq C\left| \O\right|$ and $M_{s,p'_2}(\O_{-r}) \leq M_{s,p'_2}(\O)$, whereas for the other two cases ($M_{s,q}(\O_{\d-r}\backslash \O_{-r})$, and $ |\O_{\d-r}\backslash\O_{-r}|)$ we can apply the following result. We have that for $f\in L^1(\O_\d)$, $f\geq 0$ in $\O_\d$, by the dominate convergence theorem,
 \begin{equation*}
     \int_{\Rn} f(x) \chi_{\O_{\d-r}\backslash \O_{-r}} \, dx\to \int_{\Rn} f(x) \chi_{\O_{\d}\backslash \O} \, dx
 \end{equation*}
 when $r\to 0^+$. Thus, for a given $r_0>0$, there exists a constant $C>0$ such that for every $r\in (0,r_0)$,
  \begin{equation*}
     \int_{\Rn} f(x) \chi_{\O_{\d-r}\backslash \O_{-r}} \, dx \leq C \int_{\Rn} f(x) \chi_{\O_{\d}\backslash \O} \, dx.
 \end{equation*}
\end{proof}
\begin{rem}
    On the assumptions of Theorem \ref{th: nonlocal divergence theorem}, by $\supp w_\d=B(0,\d)$ and Lemma \ref{lem: controlling the boundary term}, so that we can swap integrals, we can write
     \begin{equation*}
     \begin{split}
         \int_{\O} \diver^s_\d \f(x) \, dx=&  -c_{n,s}\int_{\Gamma_{-\d}}\int_{\Gamma_{\d}}\f(y) \cdot \frac{x-y}{|x-y|^2}\frac{w_\d(x-y)}{|x-y|^{n+s-1}} \, dy \,dx \\
         &+c_{n,s}\int_{\Gamma_{\d}} \int_{\Gamma_{-\d}}\f(y)\cdot\frac{x-y}{|x-y|^2}\frac{w_\d(x-y)}{|x-y|^{n+s-1}} \, dy \,dx ,
          \end{split}
    \end{equation*}
    where the term $\f(y)$ can be taken outside the inner integral after swapping integrals again.
\end{rem}
Similarly, the next result can be stated.
\begin{theo} \label{th: nonlocal Stokes theorem}
	Let $n=3$, $0\leq s<1$, $1<p<\infty$ and $\e>0$. Assume $\f\in H^{s,p,\d}(\O,\Rn)\cap L^q(\Gamma_{\pm\e},\Rn$) for some $q>\frac{1}{1-s}$, then the following equality holds 
	\begin{equation} \label{eq: nonlocal Stokes theorem}
		\int_{\O} \curl^s_\d \f(x) \, dx= -c_{n,s}\int_{\Gamma_{\d}} \int_{\O} \frac{\f(x)+\f(y)}{|x-y|}\times\frac{x-y}{|x-y|}\frac{w_\d(x-y)}{|x-y|^{n+s-1}} \, dx \,dy .
	\end{equation}
\end{theo}
\begin{proof}
	By odd symmetry we have that
	\begin{equation*}
		\begin{split}
			\curl^s_\d \f(x)&=c_{n,s} \int_{B(x,\d)} \frac{\f(x)-\f(y)}{|x-y|} \times \frac{x-y}{|x-y|} \frac{w_\d(x-y)}{|x-y|^{n-1+s}} \, dy \\
			&= -c_{n,s} \pv_x\int_{B(x,\d)} \frac{\f(x)+\f(y)}{|x-y|} \times \frac{x-y}{|x-y|} \frac{w_\d(x-y)}{|x-y|^{n-1+s}} \, dy		.
		\end{split}
	\end{equation*}
	Then, following the same steps from the proof of Theorem \ref{th: nonlocal divergence theorem} and placing the cross product instead of the dot product, it yields \eqref{eq: nonlocal Stokes theorem}.
\end{proof}
\subsection[Integration by parts]{Integration by parts}

This section is devoted to different formulations of the nonlocal integration by parts. As a difference to \cite[Th. 3.2]{BeCuMC22}, the main proof relies on the nonlocal Leibniz rule, Lemma \ref{lem: Divergencia no local producto}, and Divergence theorem, i.e, Theorem \ref{th: nonlocal divergence theorem}. As no function is assumed to have compact support in $\O$, two boundary terms that can be seen as an entangled one, appear. Such result will be later on extended by density to the nonlocal spaces $H^{s,p,\d}(\O)$. However, a bit of extra integrability of the functions on a neighbourhood of the boundary is required to complete the proof. Finally, a third formulation of the integration by parts shows how the boundary term can be rearranged as one over the double collar $\Gamma_{\pm \d}$, showing more resemblance to the classical case, but involving a principal value integral.

The results in this section remain valid for $\O$ unbounded as long as its boundary $\partial \O$ lies on a bounded domain. The case $\O=\Rn$ would be completely analogous to \cite[Lem. 2.6]{BeCuMC21}.

\begin{theo}[\textbf{Integration by parts 1}]\label{th: NL Integration by parts}
	Let $0\leq s<1$ and $0<\d$. Suppose that $u \in C^\infty_c(\Rn)$ and $\f \in C^\infty_c(\Rn,\Rn)$. Then 
	\begin{equation} \label{eq: nonlocal intergration by parts}
		\begin{split}
			\int_{\O}  D^s_\d u(x) \cdot \f(x) \, dx=& -\int_{\O} u(y) \diver_\d^s \f(y) \, dy \\
			&-	 c_{n,s} \int_{\Gamma_{\d}}\int_{\O} \frac{u(y)\f(x) + u(x)\f(y)}{|y-x|} \cdot \frac{y-x}{|y-x|}\frac{w_{\d}(x-y)}{|y-x|^{n-1+s}} \, dy \, dx.
		\end{split}
	\end{equation}
\end{theo}
\begin{proof}
	We want to use the nonlocal Leibniz formula, which means that we have to make the term $K^{s,\d}$ appears. Therefore, using Fubini's theorem and adding the necessary terms it yields
	\begin{align*}
		\int_{\O}  D^s_\d u(x) \cdot \f(x) \, dx=&c_{n,s}\int_{\O_\d} \int_{\O} \frac{u(x)-u(y)}{|x-y|} \frac{x-y}{|x-y|} \frac{w_{\d}(x-y)}{|x-y|^{n-1+s}} \cdot \f(x)\, dx \, dy \\
		=&\int_{\O} K_{u}^{s,\d} (\f)(y) \, dy 
		+c_{n,s}\int_{\Gamma_\d} \int_{\O} \frac{u(x)-u(y)}{|x-y|}\f(x)\cdot \frac{x-y}{|x-y|}\frac{w_{\d}(x-y)}{|x-y|^{n-1+s}} \, dx \, dy \\
		&	-c_{n,s}\int_{\O} \int_{\Gamma_\d} \frac{u(x)-u(y)}{|x-y|}\f(x)\cdot \frac{x-y}{|x-y|} \frac{w_{\d}(x-y)}{|x-y|^{n-1+s}} \, dx \, dy.
	\end{align*}
	If we use now the nonlocal Leibniz rule (Lemma \ref{lem: Divergencia no local producto}), so that we can substitute the term $K^{s,\d}_u$, we have that
	\begin{align*}
		\int_{\O} D^s_\d u(x) \cdot \f(x) \, dx=&-\int_{\O} u(y) \diver_\d^s \f(y) \, dy  +\int_{\O} \diver_\d^s (u\f)(y) \, dy \\
		&+c_{n,s}\int_{\Gamma_{\d}} \int_{\O} \frac{u(x)-u(y)}{|x-y|}\f(x)\cdot \frac{x-y}{|x-y|}\frac{w_{\d}(x-y)}{|x-y|^{n-1+s}} \, dx \, dy \\
		&	-c_{n,s}\int_{\O} \int_{\Gamma_{\d}} \frac{u(x)-u(y)}{|x-y|}\f(x)\cdot \frac{x-y}{|x-y|}\frac{w_{\d}(x-y)}{|x-y|^{n-1+s}} \, dx \, dy.
	\end{align*}
	By the nonlocal divergence theorem (Theorem \ref{th: nonlocal divergence theorem}),
	\begin{equation*}
		\int_{\O} \diver_\d^s (u\f)(y) \, dy=-c_{n,s}  \int_{\Gamma_{\d}} \int_{\O} \frac{ u(y)\f(y) + u(x)\f(x)}{|y-x|} \cdot \frac{y-x}{|y-x|}\frac{w_{\d}(y-x)}{|y-x|^{n-1+s}} \, dy \, dx.
	\end{equation*}
	Let us consider now the three boundary terms from the last two equations. In order to add them we unify the notation of the integration variables, thus, let $x$ corresponds to the integration variable over $\O_{B,\d}$ and $y$ the one corresponding to $\O$.
	If we add the three boundary terms from the last two equations we obtain that
	\begin{align*}
		&-\int_{\Gamma_{\d}} \int_{\O} \frac{ u(y)\f(y) + u(x)\f(x)}{|y-x|} \cdot \frac{y-x}{|y-x|}\frac{w_{\d}(y-x)}{|y-x|^{n-1+s}} \, dy \, dx\\
		&\qquad\qquad\qquad\qquad+\int_{\Gamma_{\d}} \int_{\O} \frac{u(y)-u(x)}{|y-x|}\f(y)\cdot \frac{y-x}{|y-x|}\frac{w_{\d}(y-x)}{|y-x|^{n-1+s}} \, dy \, dx \\
			&\qquad\qquad\qquad\qquad-\int_{\O} \int_{\Gamma_{\d}} \frac{u(x)-u(y)}{|x-y|}\f(x)\cdot \frac{x-y}{|x-y|} \frac{w_{\d}(x-y)}{|x-y|^{n-1+s}} \, dx \, dy\\
			&\qquad=-\int_{\Gamma_{\d}} \int_{\O} \frac{ u(y)\f(x) + u(x)\f(y)}{|y-x|} \cdot \frac{y-x}{|y-x|}\frac{w_{\d}(y-x)}{|y-x|^{n-1+s}} \, dy \, dx.
	\end{align*}
	Consequently,
	\begin{align*}
		\int_{\O}  D^s_\d u(x) \cdot \f(x) \, dx=& -\int_{\O} u(y) \diver_\d^s u(y) \, dy \\
		&- c_{n,s}	\int_{\Gamma_{\d}} \int_{\O} \frac{u(y)\f(x) + u(x)\f(y)}{|y-x|} \cdot \frac{y-x}{|y-x|}\frac{w_{\d}(x-y)}{|y-x|^{n-1+s}} \, dy \, dx.
	\end{align*}
\end{proof}
\begin{rem}
	As a consequence of Lemma \ref{lem: controlling the boundary term}, the integration order from the boundary term is interchangeable.
\end{rem}
As happened with the divergence theorem, we state the integration by parts in $H^{s,p,\d}(\O)$, but still requiring a little bit more of integrability of the functions on a neighbourhood of the boundary of $\O$.
\begin{theo}[\textbf{Integration by parts 2}] \label{th: NL integration by parts 2}
	Let $0\leq s<1$, $1<p<\infty$ and $\e>0$. Assume $u\in H^{s,p,\d}(\O)\cap L^{q_1}(\Gamma_{\pm\e},\Rn$) for some $q_1>\frac{1}{1-s}$ and $\f\in H^{s,p',\d}(\O,\Rn)\cap L^{q_2}(\Gamma_{\pm\e},\Rn$) for some $q_2>\frac{1}{1-s}$. Then, the following equality holds. 
	\begin{equation} \label{eq: nonlocal intergration by parts Hsp}
		\begin{split}
			\int_{\O}  D^s_\d u(x) \cdot \f(x) \, dx=& -\int_{\O} u(y) \diver_\d^s \f(y) \, dy \\
			&-	 c_{n,s} \int_{\Gamma_{\d}}\int_{\O} \frac{u(y)\f(x) + u(x)\f(y)}{|y-x|} \cdot \frac{y-x}{|y-x|}\frac{w_{\d}(y-x)}{|y-x|^{n-1+s}} \, dy \, dx.
		\end{split}
	\end{equation}
\end{theo}
\begin{proof}
	Let $0<r<\e$, then, by Proposition \ref{prop sequence with mollifiers} we can take a sequence $\{u_j\}_{j\in \N}\subset C^{\infty}_c(\Rn)$ such that
	\begin{equation*}
		u_j \to u \text{ in } L^{q_1}(\Gamma_{\pm r}), \,\, u_j \to u \text{ in } L^p(\O_{\d-r}), \quad \text{ and } \quad D^s_\d u_j \to D^s_\d u \text{ in }  L^p(\O_{-r}, \Rn).
	\end{equation*}
	Similarly, we can take a sequence $\{\f_j\}_{j\in \N}\subset C^{\infty}_c(\Rn)$ such that
	\begin{equation*}
		\f_j \to \f \text{ in } L^{q_2}(\Gamma_{\pm r},\Rn), \,\, \f_j \to \f \text{ in } L^{p'}(\O_{\d-r},\Rn), \quad \text{ and } \quad D^s_\d \f_j \to D^s_\d \f \text{ in }  L^{p'}(\O_{-r}, \Rnn).
	\end{equation*}
	Next we apply Theorem \ref{th: NL Integration by parts} to the smooth functions $u_j$ and $\f_j$ on the domain $\O_{-r}$.
	\begin{equation} \label{eq: NL integration by parts for sequences}
			\begin{split}
			\int_{\O_{-r}}  D^s_\d u(x) \cdot \f(x) \, dx=& -\int_{\O_{-r}} u(y) \diver_\d^s \f(y) \, dy \\
			&-	 c_{n,s}\int_{\O_{\d-r}\backslash \O{-r}} \int_{\O} \frac{u(y)\f(x) + u(x)\f(y)}{|y-x|} \cdot \frac{y-x}{|y-x|}\frac{w_{\d}(x-y)}{|y-x|^{n-1+s}} \, dy \, dx.
		\end{split}
	\end{equation}
	As a consequence of the previous convergences we have
	\begin{equation}\label{eq:DsdivL1}
		D^s u_j \cdot \f_j \to D^s u \cdot \f \quad \text{and} \quad u_j \diver^s \f_j \to u \diver^s \f \qquad \text{in } L^1 (\O_{-r}) .
	\end{equation}
By Lemma \ref{lem: controlling the boundary term} we can write the boundary term as 
\begin{equation*}
	 \int_{\O_{\d-r} \backslash \O_{-r}}\left[ \f_j(x) \mathcal{N}u_j(x) +u_j(x) \mathcal{N}\f_j(x) \right] \, dx,
\end{equation*}
 where $\mathcal{N}$ is the operator defined in Definition~\ref{def: NL normal direction}.	In addition, thanks to the same Lemma we have that
	\begin{equation} \label{eq: boundary term convergences}
		\begin{split}
			&\mathcal{N}u_j \to \mathcal{N}u \quad \text{ in } L^q((\O_{\d-r} \backslash\O_{-r})\cap \Gamma_{\pm\e}) \cap L^p(\O_{\d-r}\backslash(\O_{-r}\cup \Gamma_{\pm\e}))  \quad \text{and} \\
			&\mathcal{N}\f_j \to \mathcal{N}\f \quad  \text{ in } L^q((\O_{\d-r} \backslash\O_{-r})\cap \Gamma_{\pm\e}) \cap L^{p'}(\O_{\d-r}\backslash(\O_{-r}\cup \Gamma_{\pm\e})) 
		\end{split}
	\end{equation}
	for every $q<\frac{1}{s}$. Hence, as the conjugate exponents of $q_1$ and $q_2$ verify $q_1'<\frac{1}{s}$ and $q_2'<\frac{1}{s}$,
	\begin{equation*}
		\f_j \mathcal{N}u_j \to \f \mathcal{N} u \quad \text{and} \quad u_j\mathcal{N}\f_j \to u \mathcal{N} \f \text{ in } L^1(\O_{\d-r} \backslash\O_{-r})
	\end{equation*}
since we can apply H\"older inequality on both sets, $(\O_{\d-r} \backslash\O_{-r})\cap \Gamma_{\pm\e}$ and $\O_{\d-r}\backslash(\O_{-r}\cup \Gamma_{\pm\e})$.
	As a consequence, we can take the limit in \eqref{eq: NL integration by parts for sequences} when $j$ goes to infinity.
	\begin{equation*}
		\begin{split}
			\int_{\O_{-r}}  D^s_\d u(x) \cdot \f(x) \, dx=& -\int_{\O_{-r}} u(y) \diver_\d^s \f(y) \, dy \\
			&- c_{n,s} 	\int_{\O_{\d-r}\backslash \O{-r}}\int_{\O} \frac{u(y)\f(x) + u(x)\f(y)}{|y-x|} \cdot \frac{y-x}{|y-x|}\frac{w_{\d}(x-y)}{|y-x|^{n-1+s}} \, dy \, dx.
		\end{split}
	\end{equation*}
	Finally, the statement follows taking $r\to 0$, as the constants from the inequalities in Lemma \ref{lem: controlling the boundary term}, $M_{s,q}(\O_{\d-r}\backslash \O_{-r})$,$M	_{s,p_2'}(\O_{-r})$, $|\O_{-r}|$, and $ |\O_{\d-r}\backslash\O_{-r}|$ can be bounded, respectively, by $M_{s,q}(\Gamma_{\d})$,	$M_{s,p_2'}(\O)$, $|\O|, \text{ and } |\Gamma_\d|$.
\end{proof}

\begin{rem}
    The additional integrability over $\Gamma_{\pm \d}$ is not required for functions with either compact support in $\O$ or the support of at least one of them is contained in $\O_{-\d}$.
\end{rem}

For a final formulation, we rearrange the boundary terms.
		\begin{theo}[\textbf{Integration by parts 3}] \label{th: nonlocal intergration by parts alt}
					Let $0\leq s<1$, $1<p<\infty$ and $\e>0$. Assume $u\in H^{s,p,\d}(\O)\cap L^{q_1}(\Gamma_{\pm\e},\Rn$) for some $q_1>\frac{1}{1-s}$ and $\f\in H^{s,p',\d}(\O,\Rn)\cap L^{q_2}(\Gamma_{\pm\e},\Rn$) for some $q_2>\frac{1}{1-s}$. Then
	\begin{equation} \label{eq: nonlocal intergration by parts alt}
 \int_{\O}  D^s_\d u(x) \cdot \f(x) \, dx=
 \int_{\Gamma_\d}\mathcal{N}\f(x)\,dx-\int_{\O_{-\d}} u(y) \diver^s_\d \f(y) \, dy 
	\end{equation}
			\end{theo}
		\begin{proof}
From Theorem \ref{th: NL integration by parts 2}, 
			\begin{equation} \label{eq: first Green identity 1}
				\begin{split}
					\int_{\O}  D^s_\d u(x) \cdot \f(x) \, dx=& -\int_{\O} u(y) \diver_\d^s \f(y) \, dy \\
					&\qquad- c_{n,s}	\int_{\Gamma_{\d}}u(x) \int_{\O} \frac{\f(y)}{|y-x|} \cdot \frac{y-x}{|y-x|}\frac{w_{\d}(y-x)}{|y-x|^{n-1+s}} \, dy \, dx\\
					&\qquad-	 c_{n,s}\int_{\Gamma_{\d}}\f(x) \int_{\O} \frac{u(y)}{|y-x|} \cdot \frac{y-x}{|y-x|}\frac{w_{\d}(y-x)}{|y-x|^{n-1+s}} \, dy \, dx.
				\end{split}
			\end{equation}
The first integral on the right-hand side may be written as
			\begin{equation*}
				\int_{\O}u(y)\diver^s_\d \f(y)\, dy =\int_{\O_{-\d}}u(y)\diver^s_\d \f(y)\, dy +\int_{\Gamma_{-\d}}u(y)\diver^s_\d \f(y)\, dy.
			\end{equation*}
		By Definition \ref{def: nonlocal gradient} and odd symmetry,
  \begin{equation} \label{eq: first Green identity 2}
			\begin{split}
				\int_{\Gamma_{-\d}}u(x)\diver^s_\d \f(x)\, dx
        =&-\int_{\Gamma_{-\d}}u(x) \pv_x\int_{\O} \frac{\f(y)}{|x-y|}\cdot\frac{x-y}{|x-y|}\frac{w_\d(x-y)}{|x-y|^{n-1+s}}\, dy\, dx\\
				&\qquad-\int_{\Gamma_{-\d}}u(x)\int_{\Gamma_{\d}}\frac{\f(y)}{|x-y|}\cdot\frac{x-y}{|x-y|}\frac{w_\d(x-y)}{|x-y|^{n-1+s}}\, dy\, dx.
			\end{split}
		\end{equation}
  Next, we use Lemma \ref{lem: controlling the boundary term}. Arguing as in the proof of Theorem \ref{th: NL integration by parts 2}, we deduce that
  \[
    x\mapsto u(x)\int_{\Gamma_{-\d}}\frac{\f(y)}{|x-y|}\cdot\frac{x-y}{|x-y|}\frac{w_\d(x-y)}{|x-y|^{n-1+s}}\, dy   
  \]
  is integrable over $\Gamma_{-\d}$, yielding
  \begin{multline} \label{eq: first Green identity 3}
      -\int_{\Gamma_{-\d}}u(x)\int_{\Gamma_{\d}}\frac{\f(y)}{|x-y|}\cdot\frac{x-y}{|x-y|}\frac{w_\d(x-y)}{|x-y|^{n-1+s}}\, dy\, dx \\
      =\int_{\Gamma_{\d}}\f(y)\int_{\Gamma_{-\d}}\frac{u(x)}{|x-y|}\cdot\frac{y-x}{|x-y|}\frac{w_\d(x-y)}{|x-y|^{n-1+s}}\, dy\, dx.
\end{multline}
  Putting together \eqref{eq: first Green identity 1}, \eqref{eq: first Green identity 2} and \eqref{eq: first Green identity 3} yields \eqref{eq: nonlocal intergration by parts alt}.
		\end{proof}

  \begin{rem}
   Rather than assuming the additional integrability over $\Gamma_{\pm \d}$, it is sufficient for there to exist $\{\f_j\}_{j\in \N} \subset C^{\infty}_c(\Rn)$ such that $\f_j \to \f$ in $H^{s,p',\d}(\O)$ and $\mathcal{N}(\f_j) \to \mathcal{N}(\f)$ in $L^{p'}(\Gamma_{\pm\d})$.
  \end{rem}

  \subsection[Nonlocal integration by parts for the $\curl^s_\d$]{Nonlocal integration by parts for the $\curl^s_\d$}

We obtain the following result applying Theorem \ref{th: nonlocal Stokes theorem} instead of Theorem \ref{th: nonlocal divergence theorem} and \eqref{eq: NL Leibniz rule curl} instead of \eqref{eq: NL Leibniz rule div} in Theorem \ref{th: NL Integration by parts} as well as writing cross product instead of dot product.
\begin{theo} \label{th: NL integration by parts curl}
	Let $0\leq s<1$, $1<p<\infty$, $\O \subset \R^3$ and $\e>0$. Assume $u\in H^{s,p,\d}(\O)\cap L^{q_1}(\Gamma_{\pm\e},\R^3$) for some $q_1>\frac{1}{1-s}$ and $\f\in H^{s,p',\d}(\O,\R^3)\cap L^{q_2}(\Gamma_{\pm\e},\R^3$) for some $q_2>\frac{1}{1-s}$. Then, the following equality holds. 
	\begin{equation} \label{eq: nonlocal intergration by parts Hsp}
		\begin{split}
			\int_{\O}  D^s_\d u(x) \times \f(x) \, dx=& -\int_{\O} u(y) \curl_\d^s \f(y) \, dy \\
			&-	 c_{n,s} \int_{\Gamma_{\d}}\int_{\O} \frac{u(y)\f(x) + u(x)\f(y)}{|y-x|} \times \frac{y-x}{|y-x|}\frac{w_{\d}(y-x)}{|y-x|^{2+s}} \, dy \, dx.
		\end{split}
	\end{equation}
\end{theo}
Rearranging the boundary terms produces
  \begin{theo} \label{th: nonlocal intergration by parts curl alt}
					Let $0\leq s<1$, $1<p<\infty$, $\O \subset \R^3$ and $\e>0$. Assume $u\in H^{s,p,\d}(\O)\cap L^{q_1}(\Gamma_{\pm\e},\R^3$) for some $q_1>\frac{1}{1-s}$ and $\f\in H^{s,p',\d}(\O,\R^3)\cap L^{q_2}(\Gamma_{\pm\e},\R^3$) for some $q_2>\frac{1}{1-s}$. Then
				\begin{equation} 
				\begin{split}
					\int_{\O}  D^s_\d u(x) \times \f(x) \, dx=&-\int_{\O_{-\d}} u(y) \curl^s_\d \f(y) \, dy \\
					&\qquad- c_{n,s}	\int_{\Gamma_\d} u(x) \int_{\O} \frac{\f(y)}{|y-x|} \times \frac{y-x}{|y-x|}\frac{w_{\d}(y-x)}{|y-x|^{2+s}} \, dy \, dx \\
					&\qquad- c_{n,s}	\int_{\Gamma_{-\d}}u(x)\pv_x\int_{\O} \frac{\f(y)}{|y-x|} \times \frac{y-x}{|y-x|}\frac{w_{\d}(y-x)}{|y-x|^{2+s}} \, dy \, dx.
				\end{split}
				\end{equation}
			\end{theo}

	\section[ Fundamental solution of the Laplacian $\Delta_\delta$]{Fundamental solution of the Laplacian $\Delta^s_\delta$} \label{se: NL fundamental solution}

Fundamental solutions to the diffusion operator provides a very practical tool in the analysis and study of partial differential equations. This section is devoted to showing existence of fundamental solutions for the nonlocal Laplacian $\Delta_\d^s$. Interestingly enough, such a fundamental solution behaves as the fundamental solution of the classical Laplacian at infinity and as the fundamental solution of the fractional Laplacian at the origin, so that, roughly speaking, it can be viewed as an interpolation of classical and fractional fundamental solutions.  

\begin{rem} \label{rem: classical fundamental solution}
	Let $\Phi$ be the fundamental solution of the classical Laplacian. We recall that
	\begin{itemize}
		\item If $n\geq 3$, $\Phi(x)=-\frac{\Gamma(\frac{n-2}{2})}{4\pi^{\frac{n}{2}}} \frac{1}{|x|^{n-2}}\, \, \, \, \, \, (=\frac{-1}{4\sqrt{\pi}} \frac{1}{|x|}\text{ if }n=3)$.
		\item If $n=2$, $\Phi(x)=\frac{1}{2\pi} \log |x|$.
		\item If $n=1$, $\Phi(x)=\frac{|x|}{2}$.
	\end{itemize}
\end{rem}
We again use $\sigma_{n-1}=2\pi^{\frac{n}{2}}/\Gamma(\frac{n}{2})=c_{n,-1}$ to denote the surface area of the unit sphere in $\Rn$.
\begin{theo} \label{th: NL Green formula Laplacian}
	There exists a fundamental solution of the nonlocal laplacian $\Delta_\d^s$, i.e., there exists a function $\Phi^s_\d \in C^{\infty}(\Rn \backslash \{0\})$  such that for each $u\in C^{\infty}_c(\Rn)$
	\begin{equation*}
		(\Delta_\d^s) \left( \Phi^s_\d *u \right)(x) =u(x)	\quad \text{ and } \quad \Delta^s_\d \Phi^s_\d=\d_0
	\end{equation*}
	in the sense of distributions, where $\d_0$ is the Dirac delta distribution. In addition, the following properties hold:
	\begin{enumerate}[a)]
		\item There exists a dimension dependent function $A_n\in C^\infty(\Rn)$ (actually, $A_n \in C_0(\Rn)$ for $n\geq 3$) such that
		\begin{equation*}
			\Phi^s_\d(x)= \begin{cases}
			A_n(x)-\frac{1}{a_0^2\gamma(2s) }\frac{1}{|x|^{n-2s}}& \text{ if } (n,s)\neq(1,\frac{1}{2}) \\
			A_1(x) + \frac{1}{a_0^2\pi }\log(|x|)& \text{ if } (n,s)=(1,\frac{1}{2}).
		\end{cases}
		\end{equation*}
		\item For any $R>0$, there exists $M>0$ such that for all $x \in B(0,R)\backslash\{0\}$,
		\[|\Phi^s_\d(x)| \leq \frac{M}{|x|^{n-2s}}.
		\]
		\item The following limit at infinity for $\Phi^s_\delta(x)$ holds: \[\lim_{|x|\to \infty} \frac{\Phi^s_\d(x)}{\Phi(x)}\in (0,+\infty).\]
		\item The Fourier transform, in the sense of distributions, of the distributional gradients of $\Phi^s_\d$ and $\nabla \Phi^s_\d$ are given by
		\begin{equation*}
			\widehat{\nabla \Phi^s_\d}(\xi)= -i\frac{\xi}{|\xi|}\frac{1}{|2\pi \xi|\widehat{Q_\d^s}(\xi)^2}=\frac{\widehat{V_\d^s}(\xi)}{\widehat{Q_\d^s}(\xi)} .
		\end{equation*}
		
	\end{enumerate}
\end{theo}

\begin{cor} \label{cor: G in Hspd}
	Let $\Phi^s_\d$ be the function from Theorem \ref{th: NL Green formula Laplacian}, then, for every $R>0$ and $1\leq p<\frac{n}{n-s}$, we find $\Phi^s_\d \in H^{s,p,\d}(B(0,R))$ and
	\begin{equation*}
		D^s_\d \Phi^s_\d=V^s_\d.
	\end{equation*}
\end{cor}
\begin{proof}
	By Theorem \ref{th: NL Green formula Laplacian} \emph{e)} and since the Fourier transform defines an isomorphism on the space of tempered distributions $\mathcal{S}'$, we have that
	\begin{equation*}
		\int_{\Rn} V^s_\d (x) \cdot \f(x) \, dx =-\int_{\Rn} \Phi^s_\d (x)\diver_\d^s \f (x)\, dx \qquad \forall \f \in \mathcal{S}(\Rn,\Rn).
	\end{equation*}
If, in particular, we take $\f \in C^{\infty}_c(B(0,R)))$ we have that
\begin{equation*}
	\int_{B(0,R)} V^s_\d (x) \cdot \f(x) \, dx =-\int_{B(0,R+\d)} \Phi^s_\d (x)\diver_\d^s \f (x)\, dx \qquad \forall \f \in \mathcal{S}(\Rn,\Rn).
\end{equation*}
By \cite[Theorem 5.9 \emph{c)}]{BeCuMC22}, we have that there exists $M>0$ such that for every $x\in B(0,R)\backslash\{0\}$, 
\begin{equation} \label{eq: local bound for V}
	|V^s_\d(x)|\leq \frac{M}{|x|^{n-s}}.
\end{equation}
Therefore, $V^s_\d=D^s_\d \Phi^s_\d \in L^1_{loc}(\Rn,\Rn)$ and hence it satisfies the definition for a weak nonlocal gradient \cite[Definition 2.5]{CuKrSc23}. As, by \eqref{eq: bound for V}, $V^s_\d \in L^p_{\loc}(\Rn,\Rn)$ for every $p<\frac{n}{n-s}$ and by Theorem \ref{th: NL Green formula Laplacian} \emph{b)}, $\Phi^s_\d \in L^p_{\loc}(\Rn)$ for every $p<\frac{n}{n-2s}$, we can conclude that $\Phi^s_\d$ belong to the distributional definition of $H^{s,p,\d}(B(0,R))$ (\cite[Definition 2.7]{CuKrSc23}) for every $p<\frac{n}{n-s}$. Finally, since $B(0,R)$ is a bounded Lipschitz domain, by \cite[Theorem 2.8]{CuKrSc23} we have that $\Phi^s_\d \in H^{s,p,\d}(B(0,R))$ for every $p<\frac{n}{n-s}$ with 
\begin{equation*}
	D^s_\d \Phi^s_\d=V^s_\d.
\end{equation*}
\end{proof}

In order to prove Theorem \ref{th: NL Green formula Laplacian}, we first find a candidate, in a certain sense, for being the fundamental solution of $\Delta_\d^s$. More concretely, in the next lemma a characterization of the Fourier transform in the sense of distributions of the gradient of such a fundamental solution, provided it exists, is given. 

\begin{lem} \label{le: candidate for the gradient of G}
	A tempered distribution $\Phi^s_\d\in \mathcal{S}'(\Rn,\Rn)$ is the nonlocal fundamental solution of $\Delta^s_\d(\cdot)$ if and only if
	\begin{equation} \label{eq: Fourier transform of gradient of Gs}
		\widehat{\nabla \Phi^s_\d}(\xi)= -i\frac{\xi}{|\xi|}\frac{1}{|2\pi\xi|\widehat{Q^s_\d}(\xi)^2}.
	\end{equation}
\end{lem}
\begin{proof}
	On the one hand, we assume firsty that $\Phi^s_\d$ is a fundamental solution of $\Delta^s_\d(\cdot)$. In such case, we apply \cite[Proposition 4.3]{BeCuMC22} and \cite[Lemma 4.2]{BeCuMC22b} by duality on test functions so we can write
	\begin{equation*}
		\delta_0=\Delta_\d^s \Phi^s_\d  = \diver_\d^s(D_\d^s \Phi^s_\d)=\diver (Q^s_\d * Q^s_\d * \nabla \Phi^s_\d) 
	\end{equation*}
	in the sense of distributions.  (Recall that the convolution by $Q^s_\d$ and multiplication by $\widehat{Q^s_\d}$ is well defined in the sense of distributions \cite[Lemma B.2]{BeCuMC22}.) Given that $\d_0= \Delta \f$ with $\f$ the fundamental solution of the classical Laplacian (see Remark \ref{rem: classical fundamental solution}) we can identify 
	\begin{equation} \label{eq: identification of the gradient of G}
		Q^s_\d * Q^s_\d * \nabla \Phi^s_\d=\nabla \f.
	\end{equation}
	If we take the Fourier transform on \eqref{eq: identification of the gradient of G}, (by \cite{BeCuMC22} we have that the convolution by $Q^s_\d$ and multiplication by $\widehat{Q^s_\d}$ is well defined in a distributional sense), we have that $\widehat{\nabla \Phi^s_\d}$ needs to satisfy
	\begin{equation*}
		\widehat{\nabla \Phi^s_\d}(\xi)=\frac{\widehat{\nabla \f}(\xi)}{\widehat{Q^s_\d}(\xi)^2}= -i\frac{\xi}{|\xi|}\frac{1}{|2\pi \xi|\widehat{Q^s_\d}(\xi)^2}
	\end{equation*}
	where the Fourier transform of $\nabla \f$ is given by \cite[Lemma B.1]{BeCuMC22}.
	
	On the other hand, assuming that \eqref{eq: Fourier transform of gradient of Gs} holds, we see that, using \cite[Proposition 4.3]{BeCuMC22} by duality
	\begin{equation*}
		\diver^s_\d D^s_\d \Phi^s_\d= \diver^s_\d (Q^s_\d * \nabla \Phi^s_\d).
	\end{equation*}
	Applying now Fourier transform it yields
	\begin{equation*}
		\mathcal{F}\left(\diver^s_\d D^s_\d \Phi^s_\d\right)(\xi)=2\pi i \xi \widehat{Q^s_\d}(\xi) \cdot \widehat{Q^s_\d}(\xi)\cdot\widehat{\nabla \Phi^s_\d}(\xi)=2\pi \xi\widehat{Q^s_\d}(\xi)^2\frac{\xi}{|\xi|}\frac{1}{|2\pi\xi|\widehat{Q^s_\d}(\xi)^2}=1=\widehat{\d_0}.
	\end{equation*}
    Recalling that the Fourier transform is bijective over the space of tempered distributions, we conclude that $\diver^s_\d D^s_\d \Phi^s_\d=\d_0$.
\end{proof}

The next proposition is fundamental to the proof of Theorem \ref{th: NL Green formula Laplacian}. It establishes the existence of $\varPsi^s_\d$ with Fourier transform coincides with that of $\nabla \Phi^s_\d$ in the previous lemma. Moreover, several essential properties on $\varPsi^s_\d$ are provided. 

\begin{prop}\label{prop: gradient of the NL Green formula}
	There exists a vector field $\varPsi^s_\d \in C^{\infty}(\Rn \backslash \{0\}, \Rn)$ such that
	\begin{equation*}
		\widehat{\varPsi^s_\d}(\xi)= -i\frac{\xi}{|\xi|}\frac{1}{|2\pi \xi|\widehat{Q}_\d^s(\xi)^2}
	\end{equation*}
	in the sense of distributions. Moreover: 
\begin{enumerate}[a)]
		\item\label{item:a} For each $x \in \Rn \setminus \{0\}$, 
	\begin{equation} \label{eq: gradient G behaviour at infinity}
		\lim_{\lambda \to \infty} \lambda^{n-1}  \varPsi^s_\d(\lambda x) = \frac{c_{n,-1}}{\|Q_\d^s\|_{L^1}^2} \frac{x}{|x|^{n}}.
	\end{equation}
  
		\item\label{item:b} There exists  $T\in C^\infty(\Rn,\Rn)$ bounded vector radial function (actually, $T \in C_0(\Rn,\Rn)$ for $n\geq 2$) such that
			\begin{equation*}
			\varPsi^s_\d(x)= \begin{cases}
				T(x) + \pv \frac{n-2s}{a_0^2\gamma(2s) |x|^{n+1-2s}}\frac{x}{|x|}& \text{ if } (n,s)\neq (1,\frac{1}{2}) \\
				T(x) + \frac{1}{a_0^2\pi} \pv \frac{1}{x}& \text{ if } (n,s)= (1,\frac{1}{2}).
			\end{cases}
		\end{equation*}
	
		\item\label{item:c} For any $R>0$ there exists $M>0$ such that for all $x \in B(0,R)\backslash\{0\}$
		\[|\varPsi^s_\d(x)| \leq \frac{M}{|x|^{n+1-2s}}.\]
	\end{enumerate}
 
\end{prop}
\begin{proof}
	We first show the existence of $\varPsi^s_\d$ as a distribution. We a little abuse of notation, we call 
	\[\widehat{\varPsi^s_\d}(\xi)= -i\frac{\xi}{|\xi|}\frac{1}{|2\pi \xi|\widehat{Q_\d^s}(\xi)^2}
	.\]
	Having in mind that Fourier transform is a bijection from the space of tempered distributions onto itself, it is enough to show that $\widehat{\varPsi^s_\d}$ is a tempered distribution. To this end, we compare it with the tempered distribution (see \cite[Lemma A.1 \emph{b)}]{BeCuMC22})
	\begin{equation} \label{eq: auxiliar Fourier function at zero}
			Y(\xi):= -\frac{i\xi}{2\pi |\xi|^2} \frac{1}{\widehat{Q_\d^s}(0)^2},
	\end{equation}
	whose inverse Fourier transform in the sense of distributions is 
	\begin{equation}\label{eq: inverse FT of Y} \mathcal{F}^{-1}(Y)=\frac{x}{|x|^n}\frac{1}{\sigma_{n-1}\widehat{Q^s_\d} (0)^2},
	\end{equation}
	\cite[Lemma B.1 \emph{c)}]{BeCuMC22}. Let $\f \in C^{\infty}_c(\Rn)$ such that $\f_{B(0,\frac{1}{4})}=1$ and $\f_{B(0,\frac{1}{2})^c}=0$. We then write
	\begin{equation}\label{eq:decomposition Z}
		\widehat{\varPsi^s_\d}(\xi)=(\widehat{\varPsi^s_\d}(\xi)-Y(\xi))\varphi(\xi) +\widehat{\varPsi^s_\d}(\xi)(1-\f(\xi)) - Y(\xi)(1-\f(\xi)) + Y(\xi),
	\end{equation}
	Now, the first term in the decomposition \eqref{eq:decomposition Z} is a tempered distribution as $(\widehat{\varPsi^s_\d}-Y)\varphi\in L^1(\Rn)$ by Lemma \ref{lem: behaviour around zero}. Second and third terms are $L^\infty(\Rn)$ functions, while $Y$ is a tempered distribution. We therefore deduce that $\widehat{\varPsi^s_\d}$ is a tempered distribution.

	Let us now prove that $\beta$ is indeed a function. Applying the Fourier transform yields
	\begin{equation}\label{eq: FT of F Phi}
		\mathcal{F}(\widehat{\varPsi^s_\d})(x)=	\mathcal{F}\left((\widehat{\varPsi^s_\d}-Y)\varphi\right)(x) + 	\mathcal{F}\left(\widehat{\varPsi^s_\d}(1-\f)\right)(x) - 	\mathcal{F}(Y(1-\f))(x) + \widehat{Y}(x).
	\end{equation}
	We next examine each of the terms one by one. First, we find  $\widehat{Y}\in C^{\infty}(\Rn\backslash\{0\})$, \cite[Lemma B.1 \emph{c)}]{BeCuMC22}. As $(\widehat{\varPsi^s_\d}-Y)\varphi$ is a distribution with compact support, its Fourier transform is a $C^{\infty}$ function\cite[Theorem 2.3.21]{Grafakos08a}. Finally, the arguments employed for $\mathcal{F}(\widehat{\varPsi^s_\d}(1-\f))$ and $\mathcal{F}(Y(1-\f))$ are similar. The term $\mathcal{F}(Y(1-\f))$ was actually treated in \cite[Ex.\ 2.4.9]{Grafakos08a} where, besides being in $C^\infty(\Rn \backslash \{0\})$, it was additionally shown to decay faster the reciprical of any polynomial. Regarding the term $\mathcal{F}(\widehat{\varPsi^s_\d}(1-\f))$, using Lemma \ref{le: bounds of G hat} instead of \cite[Lemma 5.7]{BeCuMC22}, we may follow the same argument at the beginning of the proof of \cite[Theorem 5.9]{BeCuMC22}. It follows that for any $M \in \N$, we find $\mathcal{F}( \partial^\a \widehat{\varPsi^s_\d}(1-\f))\in C^M(\Rn)$ for sufficiently large $\a$. Consequently, $(2\pi i \xi)^\a\mathcal{F}(\widehat{\varPsi^s_\d}(1-\f))$ is bounded, which means that $\mathcal{F}(\widehat{\varPsi^s_\d}(1-\f)) \in C^{\infty}(\Rn\backslash \{0\})$ and decreases faster than the inverse of any polynomial. Collecting the results for the four terms leads to $\varPsi^s_\d \in C^{\infty}_c(\Rn\backslash\{0\},\Rn)$. Moreover, as $\widehat{\varPsi^s_\d}$ is vector radial, its (inverse) Fourier transform, $\varPsi^s_\d$, is vector radial.
	
	Now we proof {\it\ref{item:a})}, i.e., we address the behaviour of $\varPsi^s_\d$ at infinity. Concretely, we are going to see that the three first terms in the decomposition of $\mathcal{F}(\widehat{\varPsi^s_\d})$ decay faster at infinity than $\frac{1}{|x|^n}$, which is actually the decay rate of  $\mathcal{F}^{-1}(Y)$. In particular, let us see that there exists $C>0$ such that
	\begin{equation} \label{eq: decay of the difference}
		\left|\mathcal{F}\left((\widehat{\varPsi^s_\d}-Y)\varphi\right)(x)\right| \leq \frac{C}{|2\pi x|^n}.
	\end{equation} 
	By Lemma \ref{lem: behaviour around zero}, we have that $\left[(\partial_i)^n	(\widehat{\varPsi^s_\d}-Y)_i\right]\varphi \in L^1(\Rn) $. Thus, its Fourier transform is in $C_0(\Rn)$ and so, by the Fourier transform of a derivative, we obtain that $(2\pi i x_i)^n \mathcal{F}\left((\widehat{\varPsi^s_\d}-Y)_i\varphi\right) \in C_0(\Rn)$. As a result, \eqref{eq: decay of the difference} follows.
	Finally, since we already know that the other two terms decrease faster than the inverse of any polynomial and that $\widehat{Y}$ decreases as $1/|x|^{n-1}$ (see \eqref{eq: inverse FT of Y} ), we have that
	\begin{align*}
		\lim_{\l \to \infty} \l^{n-1} \varPsi^s_\d(\l x)=&\lim_{\l \to \infty} \lambda^{n-1} \mathcal{F}^{-1}(Y)(\l x)\\
		=& \lim_{\l \to \infty} \lambda^{n-1}\frac{1}{\widehat{Q_\d^s}(0)^2\sigma_{n-1}} \frac{\l x}{|\l x|^n}=\frac{c_{n-1}}{\|Q_\d^s\|_{L^1(\Rn)}^2} \frac{ x}{| x|^n},
	\end{align*}
	where we have used that $\widehat{Q_\d^s}(0)=\|Q_\d^s\|_{L^1(\Rn)}$ (\cite[Proposition 5.2 \emph{a)}]{BeCuMC22}) and that $c_{n,-1}=\frac{1}{\sigma_{n-1}}$.

	In order to show {\it\ref{item:b})}, we compare now $\varPsi^s_\d $ with the vectorial Riesz potential (see Lemma \ref{lemma: Fourier transform singular vector Riesz potential}), i.e., we write
	\begin{equation} \label{eq: rewriting Phi hat}
		\widehat{\varPsi^s_\d}(\xi)= -i\frac{\xi}{|\xi|}\frac{1}{|2\pi \xi| \widehat{Q_\d^s}(\xi)^2} - \frac{-i\xi}{a_0^2|\xi|}\frac{}{|2\pi \xi|^{-1+2s}} + \frac{-i\xi}{a_0^2|\xi|}\frac{1}{|2\pi \xi|^{-1+2s}}.
	\end{equation}
	Let us focus on the term
	\begin{equation*}
		\widehat{T}(\xi):= -i\frac{\xi}{|\xi|}\left(\frac{1}{|2\pi \xi| \widehat{Q_\d^s}(\xi)^2} - \frac{1}{a_0^2|2\pi \xi|^{-1+2s}}\right),
	\end{equation*}
	then
	\begin{equation} \label{eq: T hat}
		i\frac{\xi}{|\xi|}\cdot\widehat{T}(\xi)=\frac{a_0^2|2\pi\xi|^{-2+2s} - \widehat{Q^s_\d}(\xi)^2}{a_0^2|2\pi \xi|^{-1+2s}\widehat{Q^s_\d}(\xi)^2}= \frac{\left(a_0|2\pi\xi|^{-1+s}-\widehat{Q^s_\d}(\xi)\right)\, \left(a_0|2\pi\xi|^{-1+s}+\widehat{Q^s_\d}(\xi)\right)}{a_0^2|2\pi \xi|^{-1+2s} \widehat{Q^s_\d}(\xi)^2}.
	\end{equation}
	Now, since $\varPsi^s_\d\in C^{\infty}(\Rn\backslash\{0\})$, then ${T}\in C^{\infty}(\Rn\backslash\{0\})$ as well, and by {\it\ref{item:a})},
	\begin{equation} \label{eq: decay of T}
		\lim_{\lambda \to \infty} \lambda^{n-1}  T_{\d}^s(\lambda x) = \frac{c_{n,-1}}{\|Q_\d^s\|_{L^1}^2} \frac{x}{|x|^{n}} .
	\end{equation}
	We show that $T\in C^{\infty}(\Rn)$ as a consequence of the integrability of $\widehat{T}$. On the one hand, we have that the asymptotic behaviour  at infinity of the quotient of two of the factors in \eqref{eq: T hat} is given by 
	\begin{equation} \label{eq: decay extra factor T}
		\lim_{|\xi| \to \infty}\frac{ \left(a_0|2\pi\xi|^{-1+s}+\widehat{Q^s_\d}(\xi)\right)}{a_0^2|2\pi \xi|^{-1+2s} \widehat{Q^s_\d}(\xi)^2}=\lim_{|\xi| \to \infty}\frac{2}{a_0}|2\pi \xi|^{2-3s}.
	\end{equation}
	Indeed, this is a consequence of $\lim_{|\xi| \to \infty} \widehat{Q^s_\d}(\xi)/|2\pi \xi|^{-(1-s)}=a_0$ (see \cite[Proposition 5.2 ]{BeCuMC22}). In addition, as $\widehat{Q^s_\d} \in C^{\infty}(\Rn)$ and $\widehat{Q^s_\d}(\xi)>0$ for every $\xi$ (see \cite[Proposition 5.5]{BeCuMC22}), this term is continuous in $\Rn \backslash \{0\}$. On the other hand, the function $\xi\mapsto\left(a_0|2\pi\xi|^{-1+s}-\widehat{Q^s_\d}(\xi)\right)$ was studied in the proof of \cite[Theorem 5.9]{BeCuMC22} and shown to decay faster than the inverse of any polynomial in infinity. Combining this with \eqref{eq: decay extra factor T}, we have that $\widehat{T}$ decay faster than the inverse of any polynomial in infinity. This yields that $(2\pi i \xi) \widehat{T}(\xi) \in L^1(\Rn)$, which implies that $\nabla T=\mathcal{F}^{-1} \left((2\pi i \xi) \widehat{T}(\xi) \right) \in C_0(\Rn,\Rnn)$. Indeed, for every multi-index $a\in \N^n$ with $|\a|\geq 1$ we have that $\xi\mapsto(2\pi i\xi)^\a \widehat{T}(\xi)\in L^1(\Rn)$, which means that $\partial^\a T=\mathcal{F}^{-1} \left((2\pi i \xi)^\a \widehat{T}(\xi) \right) \in C_0$, and hence, $T\in C^{\infty}(\Rn,\Rn)$. In particular, as a consequence of \eqref{eq: decay of T} and the smoothness of $T$, we have that $T$ is bounded.
	
	As a consequence, by Lemma \ref{lemma: Fourier transform singular vector Riesz potential}, 
	 we have that \eqref{eq: rewriting Phi hat} is the Fourier transform of 
	\begin{equation*}
		\varPsi^s_\d(x)= \begin{cases}
			 T(x) + \pv \frac{n-2s}{a_0^2\gamma(2s) |x|^{n+1-2s}}\frac{x}{|x|}& \text{ if } (n,s)\neq (1,\frac{1}{2}) \\
			  T(x) + \frac{1}{a_0^2\pi} \pv \frac{1}{x}& \text{ if } (n,s)= (1,\frac{1}{2})
		\end{cases}
	\end{equation*}
	which shows {\it\ref{item:b})}.
	
	Lastly, in order to show {\it\ref{item:c})}, given $R>0$, for every $x\in B(0,R)\backslash\{0\}$, we have that
	\begin{equation*}
		|\varPsi^s_\d(\xi)|\leq \|T\|_{L^\infty(B(0,R))} + \frac{n-2s}{a_0^2 \gamma(2s)|x|^{n+1-2s}} \leq \left( \|T\|_{L^\infty(\Rn)} R^{n+1-2s}  + \frac{n-2s}{a_0^2\gamma(2s)} \right) \frac{1}{|x|^{n+1-2s}}.
	\end{equation*}

\end{proof}

\begin{proof}[Proof of Theorem \ref{th: NL Green formula Laplacian}]
	As $\varPsi^s_\d$ is vector radial, since its Fourier transform is so, there exists functions ({\it radial representatives}) $\bar{T},\, \bar{\varPsi}^s_\d:\R^+ \to \R$ such that, by Proposition \ref{prop: gradient of the NL Green formula} \emph{b)},
	\[
	\varPsi^s_\d=\bar{\varPsi}^s_\d(|x|)\frac{x}{|x|}= \begin{cases}
		 \bar{T}(|x|)\frac{x}{|x|}+\pv\frac{(n-2s)x}{a_0^2\gamma(2s) |x|^{n+2-2s}}& \text{ if } (n,s)\neq(1,\frac{1}{2}) \\
		\bar{T}(|x|)\frac{x}{|x|}+\frac{1}{a_0^2\pi} \pv \frac{1}{x}& \text{ if } (n,s)=(1,\frac{1}{2}).
	\end{cases} 
	\]
	Hence, we can obtain $\Phi^s_\d$ by integrating the radial representative of $\varPsi^s_\d$, $\bar{\varPsi}^s_\d$. 
 Thus, according to \eqref{eq: decay of T} and, in particular, the integrability of $T$, we define $A_n$ as 
	\begin{equation*}
 A_n(x)=\left\{ \begin{array}{ll}
	\displaystyle\int_{0}^{|x|} \bar{T}(r)dr, & \mbox{if }n=1, \\
\displaystyle\int_{1}^{|x|} \bar{T}(r)dr, & \mbox{if }n=2\\ \displaystyle\int_{|x|}^{+\infty} \bar{T}(r)dr, & \mbox{if }n\ge 3.
 \end{array}\right.
	\end{equation*}
	The smoothness of $T$ and \eqref{eq: decay of T} implies that $A_n,A_2$ and $A_1$ are $C^{\infty}(\Rn)$, whereas the asymptotic behaviour at infinity will be detailed below. Then, by integration of $\bar{\beta}$, we can define
	\begin{equation} \label{equation for G}
		\Phi^s_\d(x)= \begin{cases}
			A_n(x)-\frac{1}{a_0^2\gamma(2s) }\frac{1}{|x|^{n-2s}}& \text{ if } (n,s)\neq(1,\frac{1}{2}) \\
			A_1(x) + \frac{1}{a_0^2\pi }\log(|x|)& \text{ if } (n,s)=(1,\frac{1}{2})
		\end{cases}
	\end{equation}
	which proves Theorem \ref{th: NL Green formula Laplacian} \emph{a)} and \emph{d)}, to wit, $\nabla \Phi^s_\d =\varPsi^s_\d$. Thus, by Lemma \ref{le: candidate for the gradient of G}, $\Phi^s_\d$ is the fundamental solution of $\Delta^s_\d(\cdot)$, i.e. $\Delta^s_\d(\Phi^s_\d)=\d_0$ in the sense of distributions, which also implies that $\Delta^s_\d (\Phi^s_\d * u)=u$. In particular, as $\nabla \Phi^s_\d =\varPsi^s_\d$, $\Phi^s_\d\in C^\infty(\Rn \backslash\{0\})$.
	
	In order to show \emph{b)}, given $R>0$, for every $x\in B(0,R)\backslash\{0\}$, we have that
	\begin{equation*}
		| \Phi^s_\d(\xi)|\leq \|A_n\|_{L^\infty(B(0,R))} + \frac{1}{a_0^2 \gamma(2s)|x|^{n-2s}} \leq \left( \|A_n\|_{L^\infty(B(0,R))} R^{n-2s}  + \frac{1}{a_0^2\gamma(2s)} \right) \frac{1}{|x|^{n-2s}}.
	\end{equation*}
 As for \emph{e)}, as we did in Lemma \ref{le: candidate for the gradient of G}, since the convolution by $Q^s_\d$ is well defined in the sense of distributions,
 \begin{align}\label{Eq:FourierNLGradG}
 	\widehat{D^s_\d \Phi^s_\d}(\xi)= \widehat{Q^s_\d * \nabla \Phi^s_\d}(\xi)=\widehat{Q^s_\d}(\xi) \widehat{\nabla \Phi^s_\d}(\xi)= -i\frac{\xi}{|\xi|}\frac{1}{|2\pi \xi|\widehat{Q_\d^s}(\xi)}
 \end{align}
 where last equality comes from Proposition \ref{prop: gradient of the NL Green formula} and the equality $\nabla \Phi^s_\d =\varPsi^s_\d$. 
	
	Finally, we are left with \emph{c)}, where the behaviour at infinity of $\Phi^s_\d$ is actually determined by that of $A_n$. As happens with the fundamental solution of the classical Laplacian, we are going to consider three different cases.
	
	\emph{Case $n\geq 3$.} The resulting limit is,
	\begin{equation} \label{eq: limit An3}
		\lim_{\l \to \infty} \l^{n-2} A_n(\l x)=\frac{c_{n,-1}}{\|Q^s_\d\|_{L^1}^2} \frac{1}{(n-2)|x|^{n-2}}.
	\end{equation}
	Indeed, for $\l>0$ we have that
	\begin{equation*}
		\l^{n-2} A_n(\l x)= \l^{n-2} \int_{\l |x|}^{\infty} \overline{T}(r)\, dr=\l^{n-2} \int_{|x|}^{\infty} \overline{T}(\l r) \l\, dr.
	\end{equation*}
	Using dominated convergence theorem thanks to \eqref{eq: decay of T}, it yields
\begin{align*}
		\lim_{\l \to \infty} \l^{n-2} A_n(\l x)&=\int_{|x|}^{\infty} \lim_{\l \to \infty}\overline{T}(\l r) \l^{n-1}\, dr\\
		&=\int_{|x|}^{\infty} \frac{c_{n,-1}}{\|Q^s_\d\|^2_{L^1}}\frac{1}{r^{n-1}} \, dr=\frac{c_{n,-1}}{\|Q^s_\d\|^2_{L^1}}\frac{1}{(n-2) |x|^{n-2}}.
	\end{align*}
    As a consequence, \eqref{eq: limit An3} holds.
    
	\emph{Case $n=2$.} By \eqref{eq: decay of T}, there exists $C>0$ such that
	\begin{equation*}
		\int_{1}^{|x|} \overline{T}(r)\, dr \geq \int_{1}^{|x|}  \frac{C}{r}\, dr=C\log(|x|).
	\end{equation*}
	 Therefore, $	\lim_{|x| \to \infty}A_2(x) =\infty=	\lim_{|x|\to \infty}\log (|x|)$. Thus, in order to apply L'Hôpital's rule we first study the limit
				\begin{equation*}
					\lim_{\l \to \infty} \frac{A_2'(\l x)}{(\log(\l |x|))'}=\lim_{\l \to \infty} \frac{\overline{T}(\l |x|) \, |x|}{\frac{1}{\l |x|} \, |x|}=\frac{c_{n,-1}}{\|Q^s_\d\|_{L^1}^2}.
				\end{equation*}
				Then, by L'Hôpital's rule we have that
				\begin{equation*}
					\lim_{\l \to \infty} \frac{A_2(\l x)}{\log(\l |x|)}=\frac{c_{n,-1}}{\|Q^s_\d\|_{L^1}^2}.
			\end{equation*}
			\emph{Case $n=1$.} Lastly, applying dominated convergence again, by \eqref{eq: decay of T} we have that
			\begin{align*}
				\lim_{\l \to \infty}\frac{ A_1(\l x)}{\l}= \lim_{\l \to \infty}\frac{1}{\l}\int_{0}^{\l|x|} \overline{T}(r)\, dr=\lim_{\l \to \infty}\int_{0}^{|x|} \overline{T}(\l r) \, dr=\int_{0}^{|x|} \frac{c_{n,-1}}{\|Q^s_\d\|_{L^1}^2}=\frac{c_{n,-1}}{\|Q^s_\d\|_{L^1}^2} |x|.
			\end{align*}
		Given \eqref{equation for G} we have that \emph{c)} is actually determined by the behaviour at infinity of $A_n$.
		\end{proof}

			\section[The Three Nonlocal Green Identities]{The three nonlocal Green identities}\label{3GIdent}

  We start with a corollary of Theorem \ref{th: nonlocal intergration by parts alt}. The particular case when $u=1$ in $\O_\d$ and $\f=D^s_\d v$ for $v\in C^{\infty}_c(\Rn)$ leads to the following.
  \begin{cor}\label{cor: divergence for laplacian}
      Given $0\leq s<1$ and $v\in C^{\infty}_c(\Rn)$,
\[
    \int_{\O_{-\d}} \Delta^s_\d v(y) \, dy=\int_{\Gamma_\pm\d}\mathcal{N}D_\d^sv(x)\,dx=
            - c_{n,s} \int_{\Gamma_{\pm\d}}
            \pv_x\int_{\O} \frac{D^s_\d v(y)}{|y-x|} \cdot \frac{y-x}{|y-x|}\frac{w_{\d}(x-y)}{|y-x|^{n-1+s}} \, dy \, dx,
\]
where $\mathcal{N}$ is the operator introduced in Definition~\ref{def: NL normal direction}.
  \end{cor}
\begin{rem}
    Since $\Gamma_{\d}\cap\O=\emptyset$, the integrand on the right is singular only for $x\in\Gamma_{-\d}$. We may therefore write
      \begin{equation*}
          \begin{split}
		\int_{\O_{-\d}} \Delta^s_\d v(y) \, dy
            = &- c_{n,s} 	\int_{\Gamma_\d}\int_{\O} \frac{D^s_\d v(y)}{|y-x|} \cdot \frac{y-x}{|y-x|}\frac{w_{\d}(x-y)}{|y-x|^{n-1+s}} \, dy \, dx \\
					&\qquad-	 c_{n,s}\int_{\Gamma_{-\d}}\pv_x\int_{\O} \frac{D^s_\d v(y)}{|y-x|} \cdot \frac{y-x}{|y-x|}\frac{w_{\d}(x-y)}{|y-x|^{n-1+s}} \, dy \, dx.\\
            =&\int_{\Gamma_\d}\mathcal{N}D_\d^sv(x)\,dx
            -	 c_{n,s}\int_{\Gamma_{-\d}}\pv_x\int_{\O} \frac{D^s_\d v(y)}{|y-x|} \cdot \frac{y-x}{|y-x|}\frac{w_{\d}(x-y)}{|y-x|^{n-1+s}} \, dy \, dx.
				\end{split}
      \end{equation*}
 In the following Green's identities, a similar formulation of the collar integrals can be made.
\end{rem}

			Another result that we can obtain as a corollary of Theorem \ref{th: nonlocal intergration by parts alt} by choosing $\f=D^s_\d v$ is the first nonlocal Green Identity.
			\begin{theo}[\textbf{Green's first identity}] \label{th NL Green first identity}
			Let $0\leq s<1$, $1<p<\infty$ and $\e>0$. Assume $u\in H^{s,p,\d}(\O)\cap L^{q_1}(\Gamma_{\pm\e},\Rn$) for some $q_1>\frac{1}{1-s}$ and $v\in C^\infty_c(\Rn)$, then		

   \begin{equation*} \label{eq: Green first identity}
					\int_{\O_{-\d}} u(y) \Delta^s_\d v(y) \, dy+\int_{\O}  D^s_\d u(x) \cdot D^s_\d v(x) \, dx
                    =\int_{\Gamma\pm\d}u(x)\mathcal{N}D_\d^sv(x)\,dx
			\end{equation*}
		\end{theo}
 \begin{rem} \label{rem: comment on the generalisation for second NL derivatives}
 Let $1<p<\infty$ and $\e>0$. In view of Theorem \ref{th: nonlocal intergration by parts alt}, these results could be generalized for $u\in H^{s,p,\d}(\O)\cap L^{q_1}(\Gamma_{\pm\e})$ and $\f\in H^{s,p',\d}(\O)\cap L^{q_2}(\Gamma_{\pm\e})$ satisfying $D^s_\d u\in H^{s,p,\d}(\O,\Rn)\cap L^{q_1}(\Gamma_{\pm\e},\Rn)$ and $D^s_\d \f\in H^{s,p',\d}(\O,\Rn)\cap L^{q_2}(\Gamma_{\pm\e},\Rn)$, for some $q_1>\frac{1}{1-s}$ and $q_2>\frac{1}{1-s}$. 
 \end{rem}
		Using the symmetry of the second integral, we may apply the first nonlocal Green identity with the roles of $u$ and $v$ switched. Taking the difference produces the second nonlocal Green identity.
			\begin{theo}[\textbf{Green's second identity}]
					Let $0\leq s<1$. Suppose that $u,v\in C^\infty_c(\Rn)$, then
   		\begin{equation*} \label{eq: Second Green  identity}
		\int_{\O_{-\d}} \left[u(y) \Delta^s_\d v(y)-v(y) \Delta^s_\d u(y)\right] \, dy
		=\int_{\Gamma_{\pm\d}}\left[u(x)\mathcal{N}D_\d^sv(x)
                -v(x)\mathcal{N}D_\d^su(x)\right]\,dx.
			\end{equation*}
			\end{theo}
			Finally, Green's first identity and the nonlocal fundamental theorem of calculus yield the following result. 
			%
			%
			\begin{theo}[\textbf{Green's third identity}]
				Given $0<s<1$ and $v\in C^\infty_c(\Rn)$,  for each $x\in\O_{-\d}$,
    \begin{equation*} \label{eq: Third Green  identity}
					\int_{\O_{-\d}} \Phi^s_\d(x-y) \Delta^s_\d v(y)\,dy- v(x)
        =\int_{\Gamma_{\pm\d}}\left[\Phi^s_\d(x-y)\mathcal{N}D^s_\d v(y)
            -v(y)\mathcal{N}D^s_\d\Phi^s_\d(x-y)\right]\,dy
			\end{equation*}
			\end{theo}
\begin{proof}
    Recalling Theorem~\ref{th: NL Green formula Laplacian} and Corollary~\ref{cor: G in Hspd},
\[
    \Phi^s_\delta\in H^{s,p,\delta}(B(0,R))\cap C^\infty(B(0,R)\setminus\{0\})\quad\text{ and }\quad D^s_\delta \Phi^s_\delta=V^s_\delta\text{ on }B(0,R),
\]
    for each $R>0$. Observe, as in Proposition~\ref{prop: basic identities}\emph{\ref{prop: basic identity c})}, that $D^s_\delta \Phi^s_\delta(x-y)=-V^s_\delta(x-y)$. Using Theorem~\ref{Th: nonlocal version of FTC}, we may write
\begin{align}
\nonumber
    -v(x)
    =&-\int_{\Rn}D^s_\delta v(y)\cdot V^s_\delta(x-y)dy
    =\int_{\Rn}D^s_\delta v(y)\cdot D^s_\delta \Phi^s_\delta(x-y)dy\\
\label{Eq:Green31}
    =&\int_{\O}D^s_\delta v(y)\cdot D^s_\delta \Phi^s_\delta(x-y)dy
        +\int_{\Rn\setminus\O}D^s_\delta v(y)\cdot D^s_\delta \Phi^s_\delta(x-y)dy.
\end{align}
We will apply Green's first identity to each of the integrals above. For the first integral, we obtain 
\begin{equation}
\label{Eq:Green32}
    \int_{\O}D^s_\delta v(y)\cdot D^s_\delta \Phi^s_\delta(x-y)dy
    =\int_{\Gamma_{\pm\d}}\Phi^s_\d\mathcal{N}D^s_\d v(y)dy-\int_{\O_{-\delta}}\Phi^s_\delta(x-y)\Delta_\delta^sv(y)dy.
\end{equation}
For the second integral in~\eqref{Eq:Green31}, we use Green's first identity with $\O$ replaced by $\R^n\setminus\O$:
\begin{multline*}
    \int_{\Rn\setminus\O}D^s_\delta v(y)\cdot D^s_\delta \Phi^s_\delta(x-y)dy\\
    =\int_{\Rn\setminus\O_{\delta}}v(y)\Delta^s_\d \Phi^s_\delta(x-y)dy
        -c_{n,s} \int_{\Gamma_{\pm\d}}v(y)\pv_y
            \int_{\Rn\setminus\O} \frac{D^s_\d \Phi^s_\d(x-z)}{|z-y|}
            \cdot\frac{z-y}{|z-y|}\frac{w_{\d}(z-y)}{|z-y|^{n-1+s}} \, dz\, dy.
\end{multline*}
Note that $\Delta^s_\d \Phi^s_\d$ is well-defined, since $x\in\O_{-\d}$ implies $y\mapsto \Phi^s_\d(x-y)$ is smooth on $\Rn\setminus\O$. In fact, in view of Theorem~\ref{th: NL Green formula Laplacian}, we find $y\mapsto\Delta_\d^s\Phi^s_\d(x-y)=0$ on $\Rn\setminus\O_\d$. Thus,
\begin{align*}
    \int_{\Rn\setminus\O}D^s_\delta v(y)\cdot D^s_\delta \Phi^s_\delta(x-y)\,dy
    &=c_{n,s}\int_{\Gamma_{\pm\d}}v(y)\pv_y\int_{\O} \frac{D^s_\d \Phi^s_\d(x-z)}{|z-y|}
        \cdot\frac{z-y}{|z-y|}\frac{w_{\d}(z-y)}{|z-y|^{n-1+s}} \, dz \, dy \\
    &\;-c_{n,s} \int_{\Gamma_{\pm\d}}v(y)\pv_y
            \int_{\Rn} \frac{D_\d^s\Phi^s_\d(x-z)}{|z-y|}
            \cdot\frac{z-y}{|z-y|}\frac{w_{\d}(z-y)}{|z-y|^{n-1+s}} \, dz\, dy.
\end{align*}
We now verify that the last integral is zero. Fix $y\in\Gamma_\d$. Using the definition of $\rho_\d$ in~\eqref{kernel} and Lemma~\ref{lem: kernel primitive}, we may rewrite the integrand as
\[
    (n+1-s)D_\d^s\Phi^s_\d(x-z)\cdot\frac{y-z}{|y-z|}\frac{\rho_\d(y-z)}{|y-z|}
    =-D_\d^s\Phi^s_\d(x-z)\cdot\nabla Q_\d^s(y-z),\quad z\in \Rn\setminus\{x,y\}.
\]
With the change of variables $z\leftrightarrow x-z$, we find that
\[
    \int_{\Rn} \frac{D_\d^s\Phi^s_\d(x-z)}{|z-y|}
            \cdot\frac{z-y}{|z-y|}\frac{w_{\d}(z-y)}{|z-y|^{n-1+s}} \, dz
    =-(D_\d^s \Phi^s_\d*\nabla Q_\d^s)(x-y).
\]
This convolution and its Fourier transform, with respect to $y$, are well-defined in the distributional sense (see~\cite{BeCuMC22}). Using~\eqref{Eq:FourierNLGradG}, we determine that $\xi\mapsto\widehat{(D^s_\d \Phi^s_\d* \nabla Q^s_\d)}(x-\xi)=1$. It follows that
\[
    \pv_y\int_{\Rn} \frac{D_\d^s\Phi^s_\d(x-z)}{|z-y|}
    \cdot\frac{z-y}{|z-y|}\frac{w_{\d}(z-y)}{|z-y|^{n-1+s}} \, dz
    =\delta_0(x-y).
\]
As $x\in\O_{-\d}\cap\Gamma_{\pm\d}=\emptyset$, we conclude that
\[
    \int_{\Rn\setminus\O}D^s_\delta v(y)\cdot D^s_\delta \Phi^s_\delta(x-y)dy
    =c_{n,s}\int_{\Gamma_{\pm\d}}v(y)\pv_y\int_{\O} \frac{D^s_\d \Phi^s_\d(x-z)}{|z-y|}
        \cdot\frac{z-y}{|z-y|}\frac{w_{\d}(z-y)}{|z-y|^{n-1+s}} \, dz \, dy
\]
Incorporating this and~\eqref{Eq:Green32} into~\eqref{Eq:Green31}, produces the desired identity.
 
\end{proof}
			
	\section[Nonlocal Helmholtz decomposition]{Nonlocal Helmholtz decomposition}\label{NGD}
	
	Finally, we devote this section to the computation of the Helmholtz decomposition in this nonlocal framework. We first derive it for smooth functions, providing a structure of the potentials determining the $\curl^s_\d$-free and $\diver^s_\d$-free components. Afterward, we provide a generalization to more general functional spaces and study certain properties potentially relevant for projections of a vector field onto one of its Helmholtz components, such as orthogonality and uniqueness. 
 
	

		\begin{theo} \label{th: NL Helmholtz decomposition}
  Let $u \in C^\infty_c(\R^3,\R^3)$. Then $u$ can be decomposed into a $\curl^s_\d$-free component plus a $\diver^s_\d$-free component:
			\begin{equation} \label{eq: NL Helmholtz decomposition}
				u(x)=D^s_\d \beta(x) - \curl^s_\d F(x) \qquad x \in \O
			\end{equation}
			where for $x\in \O_\d$
\[
    \beta(x)=\int_{\O_{-\d}}  \diver^s_\d u (y) \Phi^s_\d(x-y) \, dy
        -\int_{\Gamma_{\pm\d}}\Phi^s_\d(x-y)\mathcal{N}u(y)\,dy
\]
and
\[
					F(x)=
    \int_{\O_{-\d}} \curl^s_\d u(y) \Phi^s_\d(x-y) \, dy - 	c_{n,s}\int_{\Gamma_{\pm\d}} \Phi^s_\d(x-y)\pv_y\int_{\O} \frac{u(z)}{|y-z|} \times \frac{y-z}{|y-z|}\frac{w_{\d}(y-z)}{|y-z|^{2+s}} \, dz \, dy
\]
		\end{theo}
		\begin{proof}
			By Theorem \ref{th: NL Green formula Laplacian}, for every $u \in C^\infty_c(\Rn)$, we have
			\begin{equation*}
				u(x)=\left(\int_{\O} u(y) \Delta_\d^s \Phi^s_\d(x-y) \, dy \right)=\Delta_\d^s  \int_{\O} u(y) \Phi^s_\d(x-y) \, dy,  \quad x \in \O.
			\end{equation*}
			Using Proposition \ref{Prop: vectorial linear identities} identity \emph{3)} we can write
			\begin{equation*}
				u(x)=D^s_\d \diver^s_\d \int_{\O} u(y) \Phi^s_\d(x-y) \, dy   -\curl_\d^s \curl^s_\d \int_{\O} u(y) \Phi^s_\d(x-y) \, dy .
			\end{equation*}
			We now focus on the first term on the right-hand side.
			\begin{equation*}
				\beta(x):=\diver^s_\d \int_{\O} u(y) \Phi^s_\d(x-y) \, dy
			\end{equation*}
			 The regularity assumed on $u$ and the fact that $D^s_\d \Phi^s_\d=V^s_\d \in L^1_{\loc}(\Rn,\Rn)$ (Corollary \ref{cor: G in Hspd} and Theorem \ref{Th: nonlocal version of FTC}) allows us to interchange the integral with that of the nonlocal derivative. Recalling Proposition \ref{prop: basic identities} \emph{\ref{prop: basic identity b})} and \emph{\ref{prop: basic identity c})} and the fact that $\Phi^s_\d$ is radial, we have
			\begin{equation*}
				\beta(x)= \int_{\O} u(y) D^s_\d \f^s_{\d,y}(x) \, dy
				= -\int_{\O} u (y) D^s_\d \f^s_{\d,x}(y) \, dy.
			\end{equation*}
		Now, Corollary \ref{cor: G in Hspd} implies $\f^s_{\d,x} \in H^{s,\bar{p},\d}(\O)$ with $1\leq\bar{p}<\frac{n}{n-s}$, $x\in \O$ and $\f^s_{\d,x} \in C^{\infty}(\Rn\backslash \{x\})$. Thus, $\f^s_{\d,x}$ verifies the hypothesis of Theorem \ref{th: nonlocal intergration by parts alt}. Consequently,
		\begin{equation*}
				\beta(x)=\int_{\O_{-\d}}  \diver^s_\d u (y) \Phi^s_\d(x-y) \, dy
        -\int_{\Gamma_{\pm\d}}\Phi^s_\d(x-y)\mathcal{N}u(y)\,dy
		\end{equation*}
  We obtain the result factoring out a minus sign on the boundary terms and substituting the definition of $\mathcal{N}u$.
  
The formula for $F$ is similarly, by applying Proposition \ref{prop: basic identities} \emph{\ref{prop: basic identity d})} and \emph{\ref{prop: basic identity c})} and Theorem \ref{th: nonlocal intergration by parts curl alt} to
   \begin{equation*}
       F(x):=\curl^s_\d \int_{\O} u(y)\Phi^s_\d (x-y) \, dy.
   \end{equation*} 
   
  
		\end{proof}
		\begin{rem}
  \begin{enumerate}[a)]
      \item Notice that the vector $y-z$ points outwards for $y \in \Gamma_{\d}$. The function $x\mapsto\Phi^s_\d(\widetilde{x})$ plays the role of $x\mapsto\f(\tilde{x})=\frac{1}{4\pi |\widetilde{x}|}$ from the classical case and $x\mapsto\f^s(\tilde{x})=\frac{k_{n,s}}{|\widetilde{x}|^{n-2s}}$ from the pure fractional case (over the whole space).

     \item By using Theorem \ref{th: NL integration by parts 2} and Theorem \ref{th: NL integration by parts curl}, instead of their alternative versions, we see that the functionals from the Helmholtz decomposition can instead be written as 
     \begin{align*}
				\beta(x)&=\int_{\O}  \diver^s_\d u (z) \Phi^s_\d(x-z) \, dz\\
				&\hspace{50pt}- c_{n,s}\int_{\Gamma_{\d}} \int_{\O} \frac{ \Phi^s_\d(x-z) u(y) +\Phi^s_\d(x-y) u(z) }{|z-y|} \cdot \frac{y-z}{|z-y|} \frac{w_\d(z-y)}{|z-y|^{2+s}}  dz \, dy
				\intertext{and}
				F(x)&=\int_{\O}  \curl^s_\d u (z) \Phi^s_\d(x-z) \, dz\\
				&\hspace{50pt}-c_{n,s}\int_{\Gamma_{\d}} \int_{\O} \frac{ \Phi^s_\d(x-z) u(y) +\Phi^s_\d(x-y) u(z) }{|z-y|} \times \frac{y-z}{|z-y|} \frac{w_\d(z-y)}{|z-y|^{2+s}} \, dz \, dy.
			\end{align*}
   This might be relevant if the boundary conditions for $u$ left unprescribed in $\Gamma_{-\d}$.
   \end{enumerate}
 \end{rem}
Adding an extra condition on the boundary allows us to extend the result for functions in $H^{s,p,\d}(\O,\R^3)$.
 	\begin{theo} \label{th: NL Helmholtz decomposition 2}
  Let $u \in H^{s,p,\d}(\O,\R^3)$ such that $\mathcal{N}u=g \in L^p(\Gamma_{\pm\d})$ and $\mathcal{N}u_j \to \mathcal{N}u$ in $L^p(\Gamma_{\pm\d})$ for a sequence $\{u_j\}_{j\in \N} \subset C^{\infty}_c(\R^3)$, $u_j \to u$ in $H^{s,p,\d}(\O)$. Then $u$ can be decomposed into a $\curl^s_\d$-free component plus a $\diver^s_\d$-free component:
			\begin{equation} \label{eq: NL Helmholtz decomposition 2}
				u(x)=D^s_\d \beta(x) - \curl^s_\d F(x) \qquad a.e. \, x \in \O
			\end{equation}
			where $\beta$ and $F$ belong to $H^{s,p,\d}(\O)$ and have the same structure of Theorem \ref{th: NL Helmholtz decomposition}.
		\end{theo}
  \begin{proof}
      Let $u\in H^{s,p,\d}(\O,\R^3)$ then, by hypothesis, there exists $\{u_j\}_{j\in \N} \subset C^{\infty}_c(\R^3,\R^3)$ such that $u_j \to u$ in $H^{s,p,\d}$. In addition, by assumption, $\mathcal{N}u_j \to \mathcal{N}u$ in $L^p(\Gamma_{\pm\d})$. Applying Theorem \ref{th: NL Helmholtz decomposition} we have that
      \begin{equation} \label{eq: NL HHD 0}
          u_j(x)=D^s_\d \beta_j(x) - \curl^s_\d F_j(x) \qquad x \in \O.
      \end{equation}
      Now, we focus on the sequences $\beta_j$ and $D^s_\d \beta_j$, in particular, on the integrability of their terms. By Young convolution inequality we have that 
      \begin{equation} \label{eq: NL HHD 1}
          \left(\int_{\O_\d} \left|\int_{\O_{-\d}}  \diver^s_\d (u_j-u) (y) \Phi^s_\d(x-y) \, dy \right|^p \, dx\right)^\frac{1}{p} \leq \|D^s_\d u_j-D^s_\d u\|_{L^p(\O,\R^3)} \|\Phi^s_\d \|_{L^1(\O_\d)} 
      \end{equation}
      where we have used that $\diver^s_\d u=\tr D^s_\d u$ (\cite[Lemma 3.5]{BeCuMC22b}). 
      The same argument applies to the boundary term, yielding
      \begin{equation} \label{eq: NL HHD 2}
          \left(\int_{\O_\d} \left|\int_{\Gamma_{\pm \d}} \Phi^s_\d(x-y)\,\left[\mathcal{N}u_j(y)-\mathcal{N}u(y) \right] \, dy\right|^p \, dx\right)^\frac{1}{p} \leq \|\mathcal{N}u_j-\mathcal{N}u \|_{L^p(\Gamma_{\pm \d})} \|\Phi^s_\d \|_{L^1(\O_\d)}
      \end{equation}
      Hence, since by Theorem \ref{th: NL Green formula Laplacian} \emph{b)} $\|\Phi^s_\d \|_{L^1(\O_\d)}$ is bounded, ${\beta_j} \to \beta$ in $L^p(\O_\d)$ where $\beta$ is defined as in Theorem \ref{th: NL Helmholtz decomposition}. As for $D^s_\d \beta_j$ we have that, 
\begin{equation} \label{eq: NL HHD 3}
          \left(\int_{\O} \left|\int_{\O_{-\d}}  \diver^s_\d (u_j-u) (y) D^s_\d\Phi^s_\d(x-y) \, dy \right|^p \, dx\right)^\frac{1}{p} \leq \|D^s_\d u_j-D^s_\d u\|_{L^p(\O,\R^3)} \|V^s_\d \|_{L^1(\O_\d)} 
      \end{equation}
      by Corollary \ref{cor: G in Hspd}. Now,  $\|V^s_\d \|_{L^1(\O_\d)}$ is bounded thanks to Theorem \ref{Th: nonlocal version of FTC}. As the argument is analogous for the boundary term, this justifies the derivation under the integral sign taking into account that $D^s_\d \beta_j =\nabla (Q^s_\d * \beta_j)$, the integrability of $Q^s_\d$ and \eqref{eq: NL HHD 1},\eqref{eq: NL HHD 2} and \eqref{eq: NL HHD 3}. Consequently, 
      \begin{equation} \label{eq: NL HHD 4}
          \beta_j - \beta \to 0 \qquad \text{ in } H^{s,p,\d}(\O).
      \end{equation}
      Therefore, by definition, $\beta \in H^{s,p,\d}(\O)$  and $\beta$ has the structure specified in Theorem \ref{th: NL Helmholtz decomposition}.

      On the other hand, as 
      \begin{equation*}
          \curl^s_\d u_j=\Large(D^s_{\d,3} -D^s_{\d,3}, D^s_{\d,1} -D^s_{\d,3} ,D^s_{\d,2} -D^s_{\d,1} \Large) ,
      \end{equation*}
      its norm is bounded by that of $\|D^s_\d u\|_{L^p(\O,\R^3)}$. Thus, we can apply the same argument obtaining that, for $F$ defined as in Theorem \ref{th: NL Helmholtz decomposition},
      \begin{equation} \label{eq: NL HHD 5}
          F_j - F \to 0 \qquad \text{ in } H^{s,p,\d}(\O)
      \end{equation}
      Then, \eqref{eq: NL Helmholtz decomposition 2} follows from \eqref{eq: NL HHD 0}, \eqref{eq: NL HHD 4} and \eqref{eq: NL HHD 5}.
  \end{proof}
  We complement the previous theorem with orthogonality and uniqueness results. To be more precise, uniqueness is established for the terms $D^s_\d \beta$ and $\curl^s_\d F$ but not necessarily for the potentials. For this, we impose the analogous condition of the divergence-free component being tangential to the boundary.
  
  \begin{prop} \label{prop: uniqueness and orthgonality of NL HHD}
      Let $u\in C^{\infty}_c(\R^3)$. 
   The Helmholtz decomposition from Theorems \ref{th: NL Helmholtz decomposition} and \ref{th: NL Helmholtz decomposition 2} 
   is orthogonal and unique provided the $\diver^s_\d$-free component verifies $\mathcal{N} \curl^s_\d F=0$ in $\Gamma_{\pm\d}$.

    \end{prop}
    \begin{proof}
        Firstly observe orthogonality, i.e.
        \begin{equation} \label{eq: NL orthoHHD}
            \int_{\O} D^s_\d \beta(x) \cdot \curl^s_\d F(x)=0.
        \end{equation}
        Indeed, by Theorem \ref{th: nonlocal intergration by parts alt} we have
        \begin{equation*}
             \int_{\O} D^s_\d \beta(x)  \cdot \curl^s_\d F(x)= \int_{\O_{-\d}}  \beta(x) \diver^s_\d \curl^s_\d F(x)=0
        \end{equation*}
        where the boundary terms vanished as a consequence of the boundary conditions and in the last step we have used Proposition \ref{Prop: vectorial linear identities} \emph{2)}.

        To prove uniqueness, we follow the standard strategy. Assume there exists $D^s_\d \beta_1$, $D^s_\d \beta_2$, $\curl^s_\d F_1$, and $\curl^s_\d F_2 $ such that
        \begin{equation*}
				D^s_\d \beta_1(x) - \curl^s_\d F_1(x) =u(x)=D^s_\d \beta_2(x) - \curl^s_\d F_2(x) \qquad x \in \O.
			\end{equation*}
   Then,
    \begin{equation} \label{eq: two HHD}
				D^s_\d (\beta_1-\beta_2)(x) - \curl^s_\d (F_1-F_2)(x) =0 \qquad x \in \O.
			\end{equation}
   Taking the inner product of \eqref{eq: two HHD} with $\curl^s_\d (F_1-F_2)(x)$, by orthogonality, we obtain
   \begin{equation*}
       0=\int_{\O} \left|\curl^s_\d (F_1-F_2)(x) \right|^2 - \left( \curl^s_\d (F_1-F_2)(x)\right) D^s_\d (\beta_1-\beta_2)(x) \, dx=\int_{\O} \left|\curl^s_\d (F_1-F_2)(x) \right|^2.
   \end{equation*}
   This implies that $\curl^s_\d F_1=\curl^s_\d F_2$, and as a consequence of \eqref{eq: two HHD}, $D^s_\d \beta_1=D^s_\d \beta_2$.
    \end{proof}
\begin{rem}
        Last result may be extended by density to functions in the Hilbert space $H^{s,2,\d}(\O)$ if we assume the hypothesis from Theorem \ref{th: NL Helmholtz decomposition 2} and, in addition, that $\mathcal{N}\curl^s_\d F_j \to \mathcal{N}\curl^s_\d F=0$ in $L^2(\Gamma_{\pm \d},\Rn)$, where $\curl^s_\d F_j$ correspond to the $\diver^s_\d$-free component of each $u_j$.
    \end{rem}
    \begin{prop} \label{prop: zero NL div}   Let $u \in H^{s,p,\d}(\O,\R^3)$ such that $\mathcal{N}u=g \in L^p(\Gamma_{\pm\d})$ and $\mathcal{N}u_j \to \mathcal{N}(u)$ in $L^p(\Gamma_{\pm\d})$ for a sequence $\{u_j\}_{j\in \N} \subset C^{\infty}_c(\R^3)$, $u_j \to u$ in $H^{s,p,\d}(\O)$. 
        Assume, in addition, $\diver^s_\d u=0$ and consider one of the following two hypothesis
    \begin{enumerate}
    \item[$(H1)$] $\mathcal{N} u=0$ in $\Gamma_{\pm\d}$, 
    \item[$(H2)$] $\mathcal{N} D^s_\d \beta=0$ in $\Gamma_{\pm\d}$ and $u\in C^{\infty}_c(\R^3,\R^3)$,
    \end{enumerate}
    where $D^s_\d \b$ is the term from Theorem \ref{th: NL Helmholtz decomposition}. Then there exists $F:\R^3 \to \R^3$ such that $u=\curl^s_\d F$ in $\O$ and $\mathcal{N} u=\mathcal{N}\curl^s_\d F$ in $\Gamma_{\pm\d}$.
    \end{prop}
    \begin{proof}
    The case assuming $(H1)$ is immediate by applying the formula provided in Theorem \ref{th: NL Helmholtz decomposition 2}, yielding $\beta=0$. Thus, let us assume $(H2)$. 
    Now, we apply $\diver^s_\d$ on $u$. Then, by the Helmholtz decomposition \eqref{eq: NL Helmholtz decomposition 2} and Proposition \ref{Prop: vectorial linear identities} \emph{2)} we have
        \begin{equation*}
            \Delta^s_\d \beta(x)=0 \qquad x \in \O_{-\d}.
        \end{equation*}
        If we multiply by $\beta$ and use the nonlocal Green's first identity (Theorem \ref{th NL Green first identity}), and that $\mathcal{N}u=0$ in $\O_\d\backslash\O_{-\d}$ so that the boundary term vanishes, then
        \begin{equation*}
            0=\int_{\O_{-\d}} \beta(x) \Delta^s_\d \beta(x)\, dx=\int_{\O} \left| D^s_\d \beta(x) \right|^2 \, dx.
        \end{equation*}
        Consequently, we obtain that $D^s_\d \beta=0$, which by \eqref{eq: NL Helmholtz decomposition}, leads us to
        \begin{equation*}
            u=\curl^s_\d F.
        \end{equation*}
    \end{proof}

\section{Conclusions}\label{Conclusions}
   The results of this paper focus on analytical studies of a nonlocal theory which was introduced in the study of nonlinear nonlocal calculus of variations \cite{BeCuMC22}, and in particular of polyconvex energy densities relvant for nonlocal nonlinear elasticity  in\cite{BeCuMC22b}. 
   The results here show a close correspondence with the classical theory, with vector identities, representation formulas, decompositions, and estimates closely resembling the differential setting. 

 The Helmholtz-Hodge decomposition has been shown to have large applicability in many fields (see the survey in \cite{bhatia2012helmholtz} on applications for the classical decomposition), so it is expected to serve as an important tool for nonlocal models that use these operators. This work shows an alternative version of it based on nonlocal operators, which might be useful for future nonlocal modelling.

    Theoretical and numerical investigations of this framework will be continued in future works and for a variety of models, where in particular we will also investigate convergence of nonlocal solutions to classical counterparts in the limit of the vanishing horizon, similar to results obtained in \cite{FossRaduYu}. Establishing well-posedness of systems with different types of boundary conditions remains an important direction of study as nonlocal models have seen successful implementations in a variety of fields and applications. Indeed, the nonlocal direction analogous of the normal direction of a vector field is addressed in this document for the first time on this nonlocal framework which, in the future, might be relevant for nonlocal Neumann boundary conditions. In particular such structure can be observed in most of the vector calculus results shown in this paper.  More technical results, such as defining trace operators for $H^{s,p,\delta}(\Omega)$ functions in the spirit of \cite{foss2021traces} will be important for theoretical and numerical studies.
    
\section*{Acknowledgements} 

\noindent The work of José Carlos Bellido has been supporetd by {\it Agencia Estatal de Investigaci\'on} (SPAIN) through grant PID2020-116207GB-I00.

\noindent The work of Javier Cueto has been supported by Fundaci\'on Ram\'on Areces. 

\noindent The work of Mikil Foss and Petronela Radu was supported by the award NSF – DMS 2109149. 


			\appendix
			
			\appendixpage
			
			\section{Auxiliary results }
 In this appendix we gather several technical results, needed for the arguments employed mainly in Section \ref{se: NL fundamental solution}. Concretely, Lemmas \ref{lem: distribution pv 1-2s} through \ref{le: bounds of G hat} are required in the proofs of Section \ref{se: NL fundamental solution}. We first introduce a technical result needed in the proofs of the divergence theorem and integration by parts to study converging sequences in $H^{s,p,\d}(\O)$, as well as in $L^q$, on a neighbourhood of the boundary of $\O$ for different integrability exponents. We recall that the result shown here is not a complete density study or another characterization of $H^{s,p,\d}(\O)$, but just a partial result approximating functions in those spaces through convolution by mollifiers over compactly contained sets in $\O$, which is enough for the analysis  in this document. For the equivalence of Definition  \ref{de: NL functional spaces density} with one based on the weak notion of the nonlocal gradient under the assumption of a Lipschitz boundary see \cite[Appendix B]{CuKrSc23}. 
 In particular, we show here that the nonlocal gradient operator can also be swapped with the convolution with a mollifier, in a domain compactly contained in $\O$. To this end, let $\eta$ be a standard mollifier, $\eta \in C^{\infty}_c(\Rn)$, $\supp \eta=B(0,\d)$, $\int_{\Rn} \eta(x)\, dx=1$ and $f\in L^1(\O)$. Then for every $x\in\O_{-\e\d}$ we define
				  \begin{equation*}
				  	f^{\e}(x):= f|_{\O}*\eta_\e(x)=\int_{\O_{-\e\d}} f(y)\eta_\e(x-y)\, dy=\int_{B(0,\e \d)} f(x-y)\eta_\e(y)\, dy,
				  \end{equation*}
			  where
				\begin{equation*}
					\eta_\e(x)=\frac{1}{\e^n}\eta\left(\frac{x}{\e}\right).
				\end{equation*}
The approximation result by convolution is as follows.
\begin{prop} \label{prop sequence with mollifiers}
Let $0<\d$, $0\leq s <1$ and $1\leq p< \infty$. Let $u\in H^{s,p,\d}(\O)$ and $\e>0$. Then
\begin{enumerate}[a)]
\item $D^s_\d (u* \eta_\e )(x)=(\eta_\e* D^s_\d u) (x)$ for $x\in \O_{-\e\d}$. \label{item mollifier nonlocal gradient}
\item $u^\e \to u$ in $L_{\loc}^p(\O_\d)$ and $D^s_\d u^\e \to D^s_\d u$ in $L^p_{\loc}(\O)$ as $\e \searrow 0$. \label{item mollifier sequence}
\end{enumerate}
\end{prop}
\begin{proof}
In order to prove \emph{\ref{item mollifier nonlocal gradient})}, we recall that by \cite[Lemma 4.2 \emph{b)}]{BeCuMC22b}, 
\begin{equation*}
D^s_\d u= \nabla (Q^s_\d *u) \quad \text{ in } \O ,
\end{equation*}
for all $u \in H^{s,p,\d}(\O)$. Since we have that for $x \in \O_{-\e\d}$
					\begin{equation*}
						Q^s_\d*(u*\eta_\e)(x)=\int_{B(0,\d)} Q^s_\d(y)\int_{B(0,\e\d)} u(x-y-z) \eta_\e(z) \,dz\, dy =(Q^s_\d *u)*\eta_\e(x)
					\end{equation*}
					(notice that $x-y-z \in \O_\d$), we can write
					\begin{equation*}
						D^s_\d (u*\eta_\e)(x)=\nabla(Q^s_\d (u*\eta_\e))(x)= \nabla((Q^s_\d *u)*\eta_\e)(x)=\nabla(Q^s_\d *u)*\eta_\e(x)=(D^s_\d u)*\eta_\e(x),
					\end{equation*}
				which shows \emph{\ref{item mollifier nonlocal gradient})}.
				
					Regarding \emph{b)}, since $D^s_\d u \in L^p(\O,\Rn)$ and $u\in L^p(\O_\d)$, for every $\omega \Subset \O$ we have that
					\begin{equation*}
						\lim_{\e \searrow 0} \|D^s_\d u^\e-D^s_\d u\|_{L^p(\omega)}=\lim_{\e \searrow 0} \|\eta_\e*D^s_\d u-D^s_\d u\|_{L^p(\omega)}=0.
					\end{equation*}
					In the same way, for every $\omega'\Subset \O_\d$ it yields
					\begin{equation*}
						\lim_{\e \searrow 0} \| u^\e- u\|_{L^p(\omega')}=0
					\end{equation*}
					and the proof is finished.
				\end{proof}
			Next we show a set of technical results employed in Section \ref{se: NL fundamental solution}. The first one is a lemma that justifies that the gradient of a certain function can be identified as a distribution, even when it is not locally integrable at zero.
				\begin{lem} \label{lem: distribution pv 1-2s}
Let $0<s<1$. 
The function $\partial_j \frac{1}{|x|^{n-2s}}$ can be identified with the tempered distribution defined componentwise as
\begin{equation}\label{eq:nablaQ}
	\left\langle \partial_j \frac{1}{|x|^{n-2s}}, \f \right\rangle = \frac{-1}{n-2s} \int_{\{ x_j>0\}} \frac{x_j}{|x|}\frac{1}{|x|^{n+1-2s}} (\f(x)-\f(-x) )\, dx , \qquad j \in \{1, \ldots, n\} 
\end{equation}
when $(n,s)\neq (1,\frac{1}{2})$. Otherwise,
\begin{equation*}
\langle \partial_j \log|x|, \f \rangle =  \int_{\{ x_j>0\}} \frac{x_j}{|x|}\frac{1}{|x|^{n+1-2s}} (\f(x)-\f(-x) )\, dx , \qquad j \in \{1, \ldots, n\} .
					\end{equation*}
				\end{lem}
				\begin{proof}
					By the chain rule we have that for $x \neq 0$, 
					\begin{equation*}
						\frac{x_j}{|x|}\frac{1}{|x|^{n+1-2s}}=
						\begin{cases}
							-(n-2s)\partial_j \frac{1}{|x|^{n-2s}} &\text{ if } n\neq2s \\
							\partial_j \log|x| &\text{ if } n=2s. 
						\end{cases}
					\end{equation*}
Assume that $n-2s\neq0$. Therefore, $\nabla \frac{1}{|x|^{n-2s}} \chi_{B(0, \e)^c}$ is in $L^1 (\Rn)$ for each $\e>0$, so it can be identified with a tempered distribution. For $j \in \{1, \ldots, n\}$ let 
\[
B^{\pm}_j (0,\e)^c=\{ x \in B(0,\e)^c : \pm x_j>0 \}.
\]
Then
					\begin{align*}
						\int_{B(0,\e)^c} \frac{x_j}{|x|}\frac{1}{|x|^{n+1-2s}}\f(x) \, dx 
						&= \int_{B^+_j (0,\e)^c} \frac{x_j}{|x|}\frac{1}{|x|^{n+1-2s}} (\f(x)-\f(-x) )\, dx .
					\end{align*}
					By the mean value theorem, 
					\[
					\left|\frac{x_j}{|x|}\frac{1}{|x|^{n+1-2s}} (\f(x)-\f(-x) ) \chi_{B^+_j (0,\e)^c} (x) \right|\leq
					\frac{2 \| \nabla \f\|_{\infty} }{|x|^{n-2s}} \chi_{B (0,\d)} (x) .
					\]
					This shows that formula \eqref{eq:nablaQ} defines a tempered distribution; moreover, by dominated convergence we obtain that
					\[
					\int_{B^+_j (0,\e)^c} \frac{x_j}{|x|}\frac{1}{|x|^{n+1-2s}} (\f(x)-\f(-x) )\, dx  \to \int_{\{ x_j>0\}} \frac{x_j}{|x|}\frac{1}{|x|^{n+1-2s}} (\f(x)-\f(-x) )\, dx  
					\]
					as $\e \to 0$. 
					
					The case $n=2s$ is completely analogous.
				\end{proof}
The previous lemma is needed for the next result where we compute the Fourier transform of a particular function, a result that we have not been able to find in the literature.
\begin{lem} \label{lemma: Fourier transform singular vector Riesz potential}
If $n \geq 2$ and $-1<\a<n-1$ or $n=1$ and $\a=-1+2s$, with $s\in (0,1)$, then
\begin{equation*}
-i \frac{\xi}{|\xi|} |2 \pi \xi|^{-\a}=\begin{cases}
	\mathcal{F} \left( \frac{n-\a-1}{\gamma(1+\a)}\frac{x}{|x|^{n-\a+1}} \right) (\xi)& \text{ if } (n,\a)\neq (1,0) \\
\mathcal{F} \left(\frac{1}{ \pi} \pv \frac{1}{x}\right)(\xi)& \text{ if } (n,\a)=(1,0).
						\end{cases}
					\end{equation*}
				\end{lem}
				\begin{proof} Let $(n,\a)\neq (1,0)$.
					Fix $j \in \{ 1, \ldots, n\}$.
					On one hand, by using Lemma \ref{lem: distribution pv 1-2s} when $n=1$ we have that
					\[
					\frac{1}{\gamma(1+\a)}\frac{\partial }{\partial x_j}\frac{1}{|x|^{n-(\a+1)}} = -\frac{n-\a-1}{\gamma(1+\a)}\frac{x_j}{|x|^{n-\a+1}}.
					\]
Thus, 
\begin{equation} \label{eq: riesz potential transform 1}
						\frac{1}{\gamma(1+\a)} \mathcal{F} \left(\frac{\partial }{\partial x_j}\frac{1}{|x|^{n-(\a+1)}}\right) (\xi) = -\frac{n-\a-1}{\gamma(1+\a)} \mathcal{F} \left( \frac{x_j}{|x|^{n-\a+1}} \right) (\xi).
					\end{equation}
On the other hand, by standard properties of the Fourier transform, and, in particular, by the Fourier transform of the Riesz potential,	\begin{equation} \label{eq: riesz potential transform 2}
		\mathcal{F} \left( \frac{\partial }{\partial x_j}\frac{1}{\gamma(1+\a)}\frac{1}{|x|^{n-(\a+1)}}\right) (\xi)= 2\pi i \xi_j \widehat{I}_{1+\a}=2\pi i \xi_j |2\pi \xi|^{-(1+\a)}=i \frac{\xi_j}{|\xi|} |2\pi \xi|^{-\a}.
					\end{equation}
By combining \eqref{eq: riesz potential transform 1} and \eqref{eq: riesz potential transform 2} we obtain the conclusion.
					
					Assume now  that $(n,\a)= (1,0)$. Then, $\frac{\xi}{|\xi|} |2\pi \xi|^{-\a}$ is the sign function, whose Fourier transform is known (see for example (\cite[Lemma B.1 \emph{c)}]{BeCuMC22})). In particular we have
					\begin{equation*}
						\mathcal{F}\left(\frac{1}{2}\frac{x}{|x|}\right)(\xi)=-i \frac{\xi}{|\xi|} \frac{1}{|2\pi \xi|}.
					\end{equation*}
					Applying Fourier transform to both terms and using that $\mathcal{F}^2(f)(x)=f(-x)$ it yields
					\begin{equation*}
						-i\frac{\xi}{|\xi|}=\mathcal{F}\left( \frac{1}{\pi} \pv \frac{1}{x} \right)(\xi).
					\end{equation*}
				\end{proof}
			In the next result we compare two functions by studying the integrability of the $n$-th derivative of their difference. Namely, we consider the function chosen as a candidate for the gradient of the fundamental solution of the nonlocal Laplacian (Section \ref{se: NL fundamental solution}) and the Riesz potential whose behavior it resembles around zero.
				\begin{lem} \label{lem: behaviour around zero}
					Let $0<s<1$ and $Y$ be the function from \eqref{eq: auxiliar Fourier function at zero}, and $\f \in C^{\infty}_c(\Rn)$ then $	\left(\widehat{\varPsi^s_\d}-Y\right)\f \in L^1(\Rn)$ and
					\[
					\left[(\partial_i)^n	\left(\widehat{\varPsi^s_\d}-Y\right)_i\right]\varphi \in L^1(\Rn)
					\]
				\end{lem}
				\begin{proof}
					\begin{equation*}
						\widehat{\varPsi^s_\d}(\xi)-Y(\xi)=-i\frac{\xi}{|\xi|}\frac{1}{|2\pi \xi|} \left( \frac{[\widehat{Q_\d^s}(0)-\widehat{Q_\d^s}(\xi)][\widehat{Q_\d^s}(0)+\widehat{Q_\d^s}(\xi)]}{\widehat{Q_\d^s}(\xi)^2\widehat{Q_\d^s}(0)^2} \right) .
					\end{equation*}
					Now, since $\widehat{Q_\d^s}$ is analytic and has a maximum at $\xi=0$ (see \cite[Prop. 5.2]{BeCuMC22}), we have that $\nabla \widehat{Q}_\d^s(0)=0$ and for $i \in \{1, \ldots, n\}$,
					\begin{equation*}
						\widehat{Q_\d^s}(0)-\widehat{Q_\d^s}(\xi)= \sum_{j=2}^\infty \frac{(\partial_i)^j \widehat{Q_\d^s}(0)}{j!}\xi_i^j.
					\end{equation*}
					As a consequence, combining that with the positivity of $Q^s_\d$ and boundedness of its derivatives \cite[Prop. 5.2 and 5.5]{BeCuMC22}, we have that $\left[(\partial_i)^n	(\widehat{\varPsi^s_\d}-Y)_i\right]\varphi \in L^1(\Rn) $.
				\end{proof}
	Finally, in order to complete the study of the function shown to be the gradient of the fundamental solution we need the following lemma where we obtain a bound for its derivatives. The proof of  this result follows the same steps as in \cite[Lemma 5.7]{BeCuMC22} for the function $V^s_\d$, with the sole difference that now $\varPsi^s_\d(\xi)=\frac{1}{\widehat{Q}^s_\d(\xi)}V^s_\d(\xi)$.
				\begin{lem} \label{le: bounds of G hat}
Let $0<s<1$. For every $\a \in \N^n$ there exists $C_\a>0$ such that for any $|\xi|\geq 1$
					\begin{equation*}
						\left|\partial^\a \widehat{\varPsi^s_\d}(\xi)\right| \leq \frac{C_\a}{|\xi|^{s(|\a|+2)-1}}.
					\end{equation*}
				\end{lem}
				\begin{proof}
					We notice we can write $\widehat{\beta}^s_\d$ as $\frac{-i\xi}{|2\pi|} f$ where
					\begin{equation*}
						f=f_1' \circ g_1, \quad f_1(t)=t^{-1}, \quad g_1(\xi)=|\xi|\widehat{Q^s_\d}(\xi).
					\end{equation*}
					Let $\gamma \in \N^n$. We recall from \cite[Lemma 5.7]{BeCuMC22} the following inequality
					\begin{equation*}
						\left|\partial^{\gamma}(f_1\circ g_1(\xi))\right| \leq \sum_{k=1}^{|\gamma|} \left| f_1^{(k)} \circ g_1(\xi) \right|	\left|G_k \right (\xi)| \leq C_{\gamma} \sum_{k=1}^{\gamma} \frac{1}{|\xi|^{s(k+1)+|\gamma|-k}} \leq \frac{C_{\gamma}}{|\xi|^{s(|\gamma|+1)}}
					\end{equation*}
					where, according to Faà di Bruno's formula, $G_k$ is the notation used in \cite[Lemma 5.7]{BeCuMC22} to refer to a linear combination of products of $k$ partial derivatives of $g_1$ whose order adds up $|\gamma|$.
					
					Based on this, one can see that the process for $f$ is very similar with the sole difference of having the derivative of $f_1$, instead of $f_1$. Since by induction we have that
					\begin{equation*}
						\left| f_1^{(k+1)} \right| =\frac{C_k}{t^{k+2}}, \quad k \in \mathbb{N}, \, \, t>0,
					\end{equation*}
					for some constants $C_k>0$, we obtain the inequality
					\begin{equation*}
						\left| (f_1')^{(k)} \circ g_1(\xi) \right|\leq \frac{C_k}{\left(|\xi|\widehat{Q}^s_\d(\xi)\right)^{k+2}} \leq \frac{C_k}{\left(|\xi|\right)^{s(k+2)}}
					\end{equation*}
					as $\left|\frac{1}{\widehat{Q}^s_\d(\xi)}\right| \leq C |\xi|^{1-s}$ (\cite[Proposition 5.2]{BeCuMC22}). 
					All in all, it yields
					\begin{equation*} 
						\left|\partial^{\gamma}(f(\xi))\right|=\left|\partial^{\gamma}(f_1'\circ g_1(\xi))\right| \leq \sum_{k=1}^{|\gamma|} \left| f_1^{(k+1)} \circ g_1(\xi) \right|	\left|G_k \right (\xi)| \leq C_{\gamma} \sum_{k=1}^{\gamma} \frac{1}{|\xi|^{s(k+2)+|\gamma|-k}} \leq \frac{C_{\gamma}}{|\xi|^{s(|\gamma|+2)}}.
					\end{equation*}
				Finally, by Leibniz's formula,
				\begin{equation*}
					\left|\partial^\a \varPsi^s_\d(\xi) \right| \leq C_\a \sum_{\b \geq \a} |\partial^\b \xi | \left| \partial^{\a-\b}f(\xi)\right| \leq \frac{C_{\a}}{|\xi|^{s(|\a|+2)-1}}.
				\end{equation*}
				\end{proof}


		\end{document}